\documentclass[11pt]{amsart}
\usepackage{mathrsfs,eucal,amsmath,amsthm,amsfonts,amssymb,amscd,amsbsy,latexsym,dsfont,stmaryrd,mathtools,enumerate,color}

\usepackage[unicode]{hyperref}
\usepackage[all,2cell]{xy} \UseAllTwocells \SilentMatrices
\usepackage[T1]{fontenc}

\begin{document}

%INVIISBLE
\newcommand{\INVISIBLE}[1]{}

%THEOREMS
\newtheorem{thm}{Theorem}[section]
\newtheorem{lem}[thm]{Lemma}
\newtheorem{cor}[thm]{Corollary}
\newtheorem{prp}[thm]{Proposition}
\newtheorem{conj}[thm]{Conjecture}

\theoremstyle{definition}
\newtheorem{dfn}[thm]{Definition}
\newtheorem{question}[thm]{Question}
\newtheorem{nota}[thm]{Notations}
\newtheorem{notation}[thm]{Notation}
\newtheorem*{claim*}{Claim}
\newtheorem{ex}[thm]{Example}
\newtheorem{rmk}[thm]{Remark}
\newtheorem{rmks}[thm]{Remarks}
\newtheorem{hyp}[thm]{Hypothesis} 
\def\labelenumi{(\arabic{enumi})}

%ARROWS
\newcommand{\aro}{\longrightarrow}
\newcommand{\arou}[1]{\stackrel{#1}{\longrightarrow}}
\newcommand{\RA}{\Longrightarrow}

%BOLDFACE AND LIKE
\newcommand{\mm}[1]{\mathrm{#1}}
\newcommand{\bm}[1]{\boldsymbol{#1}}
\newcommand{\bb}[1]{\mathbf{#1}}

\newcommand{\bA}{\boldsymbol A}
\newcommand{\bB}{\boldsymbol B}
\newcommand{\bC}{\boldsymbol C}
\newcommand{\bD}{\boldsymbol D}
\newcommand{\bE}{\boldsymbol E}
\newcommand{\bF}{\boldsymbol F}
\newcommand{\bG}{\boldsymbol G}
\newcommand{\bH}{\boldsymbol H}
\newcommand{\bI}{\boldsymbol I}
\newcommand{\bJ}{\boldsymbol J}
\newcommand{\bK}{\boldsymbol K}
\newcommand{\bL}{\boldsymbol L}
\newcommand{\bM}{\boldsymbol M}
\newcommand{\bN}{\boldsymbol N}
\newcommand{\bO}{\boldsymbol O}
\newcommand{\bP}{\boldsymbol P}
\newcommand{\bY}{\boldsymbol Y}
\newcommand{\bS}{\boldsymbol S}
\newcommand{\bX}{\boldsymbol X}
\newcommand{\bZ}{\boldsymbol Z}

%CALLIGRAPHIC
\newcommand{\cc}[1]{\mathscr{#1}}
\newcommand{\ccc}[1]{\mathcal{#1}}

\newcommand{\ca}{\cc{A}}

\newcommand{\cb}{\cc{B}}

\newcommand{\cC}{\cc{C}}

\newcommand{\cd}{\cc{D}}

\newcommand{\ce}{\cc{E}}

\newcommand{\cf}{\cc{F}}

\newcommand{\cg}{\cc{G}}

\newcommand{\ch}{\cc{H}}

\newcommand{\ci}{\cc{I}}

\newcommand{\cj}{\cc{J}}

\newcommand{\ck}{\cc{K}}

\newcommand{\cl}{\cc{L}}

\newcommand{\cm}{\cc{M}}

\newcommand{\cn}{\cc{N}}

\newcommand{\co}{\cc{O}}

\newcommand{\cp}{\cc{P}}

\newcommand{\cq}{\cc{Q}}

\newcommand{\cR}{\cc{R}}

\newcommand{\cs}{\cc{S}}

\newcommand{\ct}{\cc{T}}

\newcommand{\cu}{\cc{U}}

\newcommand{\cv}{\cc{V}}

\newcommand{\cy}{\cc{Y}}

\newcommand{\cw}{\cc{W}}

\newcommand{\cz}{\cc{Z}}

\newcommand{\cx}{\cc{X}}

%FRAKTUR
\newcommand{\g}[1]{\mathfrak{#1}}

%BOURBAKI STYLE
\newcommand{\af}{\mathds{A}}
\newcommand{\PP}{\mathds{P}}

\newcommand{\GL}{\mathrm{GL}}
\newcommand{\PGL}{\mathrm{PGL}}
\newcommand{\SL}{\mathrm{SL}}
\newcommand{\NN}{\mathds{N}}
\newcommand{\ZZ}{\mathds{Z}}
\newcommand{\CC}{\mathds{C}}
\newcommand{\QQ}{\mathds{Q}}
\newcommand{\RR}{\mathds{R}}
\newcommand{\FF}{\mathds{F}}
\newcommand{\DD}{\mathbf{D}}
\newcommand{\VV}{\mathds{V}}
\newcommand{\HH}{\mathds{H}}
\newcommand{\MM}{\mathds{M}}
\newcommand{\OO}{\mathds{O}}
\newcommand{\LL}{\mathds L}
\newcommand{\BB}{\mathds B}
\newcommand{\kk}{\mathds k}
\newcommand{\bs}{\mathbf S}
\newcommand{\GG}{\mathds G}
\newcommand{\WW}{\mathds W}
%GREEK
\newcommand{\al}{\alpha}

\newcommand{\be}{\beta}

\newcommand{\ga}{\gamma}
\newcommand{\Ga}{\Gamma}

\newcommand{\om}{\omega}
\newcommand{\Om}{\Omega}

\newcommand{\vt}{\vartheta}
\newcommand{\te}{\theta}
\newcommand{\Te}{\Theta}

\newcommand{\ph}{\varphi}
\newcommand{\Ph}{\Phi}

\newcommand{\ps}{\psi}
\newcommand{\Ps}{\Psi}

\newcommand{\ep}{\varepsilon}

\newcommand{\vr}{\varrho}

\newcommand{\de}{\delta}
\newcommand{\De}{\Delta}

\newcommand{\la}{\lambda}
\newcommand{\La}{\Lambda}

\newcommand{\ka}{\kappa}

\newcommand{\si}{\sigma}
\newcommand{\Si}{\Sigma}

\newcommand{\ze}{\zeta}

%VARIOUS
\newcommand{\mc}{\g{C}}
\newcommand{\mcrs}{\g{C}^{\rm rs}}
\newcommand{\fr}[2]{\frac{#1}{#2}}
\newcommand{\vs}{\vspace{0.3cm}}
\newcommand{\na}{\nabla}
\newcommand{\pd}{\partial}
\newcommand{\po}{\cdot}
\newcommand{\met}[2]{\left\langle #1, #2 \right\rangle}
\newcommand{\rep}[2]{\mathrm{Rep}_{#1}(#2)}
\newcommand{\repo}[2]{\mathrm{Rep}^\circ_{#1}(#2)}
\newcommand{\hh}[3]{\mathrm{Hom}_{#1}(#2,#3)}
\newcommand{\modules}[1]{#1\text{-}\mathbf{mod}}
\newcommand{\Modules}[1]{#1\text{-}\mathbf{Mod}}
\newcommand{\coh}[1]{#1\,\text{--}\,\mathbf{coh}}
\newcommand{\dmod}[1]{\mathcal{D}(#1)\text{-}{\bf mod}}
\newcommand{\spc}{\mathrm{Spec}\,}
\newcommand{\an}{\mathrm{an}}

%POWER SERIES
\newcommand{\pos}[2]{#1\llbracket#2\rrbracket}

%LAURENT SERIES
\newcommand{\lau}[2]{#1(\!(#2)\!)}

%CONVERGENT POWER SERIES
\newcommand{\cpos}[2]{#1\langle#2\rangle}

%IDENTITY
\newcommand{\id}{\mathrm{id}}

%IDENTITY IN A CATEGORY
\newcommand{\one}{\mathds 1}

%PRODUCT
\newcommand{\ti}{\times}
\newcommand{\tiu}[1]{\underset{#1}{\times}}

%TENSOR PRODUCT
\newcommand{\ot}{\otimes}
\newcommand{\otu}[1]{\underset{#1}{\otimes}}

%OVER/UNDER SYMBOLS
\newcommand{\wh}{\widehat}
\newcommand{\wt}{\widetilde}
\newcommand{\ov}[1]{\overline{#1}}
\newcommand{\un}[1]{\underline{#1}}

% DIRECT PRODUCT
\newcommand{\op}{\oplus}

%CATEGORICAL LIMITS
\newcommand{\lid}{\varinjlim}
\newcommand{\lip}{\varprojlim}

%EGA
\newcommand{\ega}[3]{[EGA $\mathrm{#1}_{\mathrm{#2}}$, #3]}

%PAGEBREAK
\newcommand{\asts}{\begin{center}$***$\end{center}}

\title[Connections on relative schemes]{Regular-singular connections on relative complex schemes}
\date{\today}

\author[P. H. Hai]{Ph\`ung H\^o Hai}

\address{Institute of Mathematics, Vietnam Academy of Science and Technology, Hanoi, 
Vietnam}

\email{phung@math.ac.vn}

\author[J. P. dos Santos]{Jo\~ao Pedro  dos Santos}

\address{Institut de Math\'ematiques de Jussieu -- Paris Rive Gauche, 4 place Jussieu, 
Case 247, 75252 Paris Cedex 5, France}

\email{joao\_pedro.dos\_santos@yahoo.com}



\maketitle

\begin{abstract} Deligne's celebrated ``Riemann--Hilbert correspondence'' appearing in \cite{deligne70} relates representations of the fundamental group of a smooth complex algebraic variety and regular-singular integrable connections. In this work, we show how to arrive at a similar statement in the case of a smooth scheme $X$ over the spectrum of a ring $R=\pos\CC {t_1,\ldots, t_r}/I$. On one side of the correspondence we have representations on $R$-modules of the fundamental group of the special fibre, and on the other we have certain integrable $R$-connections admitting logarithmic models. The correspondence is then applied to give explicit examples of differential Galois groups of $\pos\CC t$--connections.
\end{abstract}

\section{Introduction}
The objective of this paper is to show how to adapt Deligne's theory relating regular--singular connections and representations of the fundamental group \cite{deligne70},   \cite{malgrange}   to the setting of schemes defined over certain {\it complete local $\CC$-algebras}. 
This is done with the purpose of determining  easily some differential Galois groups (to be taken in the sense of   \cite{duong-hai-dos_santos18} and \cite{hai-dos_santos19}). 

Deligne's celebrated ``Riemann--Hilbert correspondence'' establishes that on a smooth complex algebraic variety $X^*$ (notation shall become clear soon),  the category of complex linear representations of the topological fundamental group $\pi_1(X^{* \an} )$
can be recovered as the category of integrable regular--singular connections. 
Although a prominent feature of this correspondence is an intrinsic definition of regular--singular connections, this shall not be discussed here. For us, regular--singular connections will only be defined after a ``compactification'' is chosen. So, let us assume that $X^*$ is  an open subvariety of a  {\it proper} and smooth complex algebraic variety $X$ (the compactification) and that $Y:=(X\smallsetminus X^*)_{\rm red}$ is a divisor with normal crossings and smooth connected   components. Then, one says that an integrable connection on $X^*$ is regular--singular whenever it might be extended to a {\it logarithmic} integrable connection on $X$. (To repeat, the appropriate terminology should be $(X,X^*)$-regular-singular as in \cite[Definition 4.1]{kindler15}, but  this is  exhausting.)
Under these definitions, we have a ``Deligne-Riemann-Hilbert'' equivalence of $\CC$-linear tensor categories 
\[
\mm{DRH}: \left\{\begin{array}{c}\text{regular--singular}\\\text{integrable connections on $X^*$}\end{array}\right\}\arou\sim\left\{\begin{array}{c}\text{finite dimensional complex}\\\text{ representations of $\pi_1(X^{*\an})$}\end{array}\right\};
\]
see for example \cite[II.5.9, p.97]{deligne70} or \cite[Theorem 7.2.1]{malgrange}. 
One relevant consequence of this theorem is that it allows one to compute differential Galois groups by means of ``Schlesinger's theorem'':  the differential Galois group is the Zariski closure of the monodromy group. 

A natural question---specially in view of our previous works \cite{duong-hai-dos_santos18} and \cite{hai-dos_santos19}---is to determine if Deligne's theory can be extended to a ``relative'' setting. So let $S$ be the spectrum of a quotient of $\pos\CC {t_1,\ldots,t_r}$ and consider a smooth and proper  $S$-scheme $X\to S$ having connected fibres. Suppose that $Y\subset X$ is a relative  divisor with  normal crossings and write $X^*$ for the complement $X\smallsetminus Y$. Defining a regular-singular connection on $X^*$ in analogy with the definition of the previous paragraph and letting $X_0^*$ be the special fibre of $X^*$,  we might enquire about the relation between 
\[
\left\{\begin{array}{c}\text{regular--singular}\\\text{integrable $S$-connections on $X^*$}\end{array}\right\}\quad \text{and}\quad  \left\{\begin{array}{c}\text{representations of $\pi_1(X_0^{*\an})$ }\\\text{ on finite $\co(S)$-modules }\end{array}\right\}.
\]
Our answer is given by Corollary \ref{25.11.2019--1} and says that under 
additional hypothesis the aforementioned categories are equivalent. 

Our proof of Corollary \ref{25.11.2019--1} follows  Deligne's original method. Let us describe our strategy rapidly while awaiting   a more precise summary  below. Firstly  a considerable part  of the work is developed in the setting of {\it smooth complex spaces} over the analytic spectrum of a $\CC$-algebra $\La$ which is finite dimensional as a vector space. For such spaces, we show that integrable $\La$-connections 
define the same objects as logarithmic {\it meromorphic} connections on a larger space (Theorem \ref{11.05.2016--2}). The key point to obtain this description is that smooth complex spaces over $\La$ {\it do not deform   locally}, so that we are able to apply Deligne's results to extend connections to logarithmic ones (see Theorem \ref{deligne_manin_extension}).     This corresponds roughly to  \cite[Theorem 5.1]{malgrange}. 
After that, we move to algebraic geometry over a {\it complete local noetherian} $\CC$-algebra $R$ with residue field $\CC$ (Section \ref{17.12.2019--1}). Employing Grothendieck's algebraization theorem  and  the fact that the order of the poles of an arrow between logarithmic connections  is   independent of the   truncation of $R$, we show how to find preferred  models for regular-singular $R$-connections (Theorem \ref{01.10.2019--1}). These findings are then assembled to obtain Corollary \ref{25.11.2019--1}, which is the main output of this work.

We now review the remaining sections       separately. In what follows,   $\La$ is a local $\CC$-algebra which is a finite dimensional complex vector space. 

Section \ref{29.01.2020--2} serves mainly to introduce notation and definitions such as:    relative divisors with normal crossings,   number of branches (Definition \ref{02.10.2019--1}) and several variations on the theme of connections  (Definitions \ref{29.01.2020--1}, \ref{04.11.2019--8} and \ref{04.11.2019--9}). This section is also the right place to talk about Deligne's notion of relative local system, a very important tool, see Theorem \ref{10.06.2016--2}. 

Section \ref{29.01.2020--3} serves two purposes and the   first   is to explain the following simple technique.   Let   $S$ be the   analytic spectrum of $\La$ and $D \subset\CC^n$ be an open subset. Then the category of $\co_{D\ti S}$-modules is canonically identified with   the category of $\co_D$-modules endowed with an action of $\La$, so that many local results from \cite[Chapter II]{deligne70}
are easily transposed to the setting of smooth complex spaces over $S$. One of these transpositions, Theorem \ref{21.10.2019--1},   is the second purpose of this section; it is yet another manifestation of the classical principle that singularities ``of the first kind'' are ``regular'' \cite[Chapter 4, Theorem 2.1]{coddington-levinson55}.   

Section \ref{18.12.2019--1} introduces the notion of {\it residues of logarithmic connections} on complex spaces over $\La$ (Definition \ref{19.12.2019--2}).
 Also in Section \ref{18.12.2019--1}  the reader shall find the definition of the {\it exponents} (Definition \ref{29.10.2019--3}) of a logarithmic connection as the spectral set of the residue morphisms. 
 (These constructions are analytic and follows the clear exposition in \cite[pp.157-8]{malgrange}.)
Key to the understanding of the theory exposed here is the fact that exponents are {\it complex numbers}, and not elements of $\La$; that this is a sensible choice is motivated by  Lemma \ref{10.10.2019--1} and Lemma \ref{06.11.2019--1}. 
 
The main result of  Section \ref{04.11.2019--4} is Theorem \ref{11.05.2016--2}. But its fundamental  pillar is Theorem  \ref{deligne_manin_extension}, which  shows that also in the case of smooth complex spaces over $\La$ (as in the previous paragraph) it is possible to extend $\La$-connections to logarithmic ones---we call these extensions {\it Deligne-Manin extensions}.  (In the literature, such extensions  go under the name of {\it canonical} or {\it $\tau$-extensions} \cite[I.4, 22ff]{andre-baldassarri01}, \cite[Definition 5.1]{kindler15}. In \cite[II.5.4]{deligne70}, these are attributed to Manin.) The proof of Theorem \ref{deligne_manin_extension} 
relies on the local existence of logarithmic extensions and their uniqueness.  
Once these   are found, it then becomes a simple matter to show  Theorem \ref{11.05.2016--2}.
 
In Section \ref{17.12.2019--1} we   turn our attention to algebraic geometry and put to use our previous findings, Serre's GAGA and Grothendieck's algebraization theorem (or as we call it,   ``GFGA'').  Section \ref{31.10.2019--1}   discusses elementary properties of relative connections and fixes terminology. Section \ref{31.01.2020--4} explains how to  extend connections on $\La$-schemes  to a previously fixed proper ambient $\La$-scheme (see Theorem \ref{22.10.2019--1});           these constructions rely on the analytic picture developed  earlier  and   GAGA. In Section \ref{31.01.2020--1}, we look at the case of a base ring $R=\pos\CC {t_1,\ldots,t_r}/I$ and, by means of the algebraization theorem, Proposition  \ref{GFGA},   plus   Proposition  \ref{20.11.2019--1} --- stating that arrows between logarithmic connections have ``bounded poles'' ---,  define better suited logarithmic extensions, see Theorem \ref{01.10.2019--1}. In Section \ref{31.01.2020--2} we assemble the various pieces and state our version of the  Deligne-Riemann-Hilbert correspondence, Corollary \ref{25.11.2019--1}.   To end, in Section \ref{31.01.2020--5},    we compute explicitly certain differential Galois groups of elementary regular-singular connections  showing that differential Galois group which fail to be of finite type abound. 

We end this introduction by calling the reader's attention to the similar-oriented works of  L. Fiorot,  T. Monteiro Fernandes and C. Sabbah (FMFS), and   N. Nitsure. 
The   context of the works of FMFS is  that of  $\cd_{M\ti T/T}$--modules, where $M$ and $T$ are complex manifolds.  A fundamental result of the theory of holonomic $\cd_M$-modules is the ``Riemann-Hilbert'' correspondence between (a derived category associated to) regular holonomic modules and (a derived category associated to) constructible  sheaves; see  \cite[Section 5.3]{kashiwara} for a precise statement. It then becomes natural to study the situation in the case ``depending on parameters''. 
In \cite[Definition 2.1]{mfs}, the authors define what it means for a holonomic $\cd_{M\ti T/T}$-module to be regular. This sets up one side of the desired correspondence by passing to the category of bounded complexes with regular cohomologies $\mm D^b_{\rm rhol}(\cd_{M\ti T/T})$.  The other side of the correspondence, the category $\mm D^b_{\CC-c}(\co_T)$, is introduced in   \cite[Definition 2.19]{mfs0}, see also \cite[\S1.2.1-3]{fmfs}. It is in  \cite{fmfs} (see their Theorem 1) that  the authors show that, under certain assumptions, we have an equivalence $\mm D^b_{\rm rhol}(\cd_{M\ti T/T})\simeq\mm D_{\CC-c}(\co_T)$. It is also interesting to note, as Sabbah explained to us,   that in the setting preferred by FMFS, the point of defining locally  the logarithm of the monodromy is based on results from Fr\'echet modules, while in our approach it comes almost without effort (see Proposition \ref{03.09.2019--1}). To conclude, the reader wishing to obtain a deeper understanding between the theory of regular-singular connections and  regular $\cd$-modules  will certainly profit from the thoughtful  paper \cite{cf}, which relates, in the case of a base-field, the existing notions. See Theorem 3.1 in \cite{cf}.    

The context of Nitsure's work is somewhat closer to ours since he pays special attention to algebraic objects and we point out preferably  \cite{nitsure1} (which  is expanded and supplemented by  \cite{nitsure2} and \cite{nitsure3}).  In it, Nitsure considers a fixed projective and smooth complex variety $V$ together with an effective divisor $D$ having only normal crossings. Then, for a complex algebraic scheme $T$, he considers relative logarithmic connections   on $V\ti T$ having poles on $D\ti T$ \cite[Definition 2.6]{nitsure1} and constructs, following the GIT strategy, a certain moduli space for those \cite[Theorem 3.5]{nitsure1}.

\subsection*{Notations and conventions}
\begin{enumerate}[1)] 

\item Complex spaces are to be understood in the sense of \cite[1.1.5]{grauert-remmert84} (called analytic spaces in \cite[Expos\'e 9, 2.1]{sc}).

\item We let   $\La$ stand for a local $\CC$-algebra whose dimension as a complex vector space is finite. Its maximal ideal is denoted by $\g m$.

\item We shall find convenient to denote by $|X|$ the topological space underlying a ringed space $X$. 

\item If $(X,\co_X)$ is a complex  space and    $p$ a point of $X$, 
we let $\g M_p$ stand for the maximal ideal of the local ring $\co_{X,p}$ (so that $\co_{X,p}=\CC\op\g M_p$).
\item If $(X,\co_X)$ is a complex   space     $p$ a point of $X$ and $\ce$ is a coherent sheaf, we let $\ce(p)$ stand for the complex vector space $\ce_p/\g M_p\ce_p$. In like manner, for a section $e\in \Ga(X,\ce)$  we let $s(p)$ be the image of $s_p\in\ce_p$ in $\ce(p)$.

\item A {\it vector bundle} over a locally ringed space $(X,\co_X)$ is simply a locally free sheaf of $\co_X$-modules of finite rank.

\item If $X$ is a scheme, respectively complex  space, $Y\subset X$ is a closed subscheme, respectively subspace,  and $\cm$ is a coherent sheaf on $X$, we write $\cm|_Y$ for the pull-back of $\cm$ to $Y$, that is, $\cm|_Y$ is an $\co_Y$-module. Analogous notations are in force for sections. 

\item If $M$ is a topological space and $E$ is a set, we let $E_M$ denote the {\it simple, or locally constant} sheaf  associated to $E$. 

\item We fix once and for all a subset $\tau$ of $\CC$  containing exactly one element   in each class of $(\CC,+)$ modulo   $(\ZZ,+)$. 

%The ring of Laurent series with coefficients in $R$, $\pos Rx[1/x]$ is denoted by the traditional $\lau Rx$. The $R$-linear derivation of  $\lau Rx$ defined by $x^k\mapsto kx^k$ will be denoted by $\vt$. Analogous notations are in force when talking about $\lau \CC x$. 

\item  If $V$ is a finite dimensional complex vector space and $A$ an endomorphism of $V$, we write $\mm{Sp}_A$ for the set of eigenvalues of $A$ and, for each $\la\in\mm{Sp}_A$, we write $\bb E(A,\la)$, respectively $\bb G(A,\la)$, for its $\la$--eigenspace, respectively generalized $\la$-eigenspace. 
 
\item We shall find convenient to work with the following version of the  Minkowiski difference of two sets of complex numbers: if $A$ and $B$ are subsets of the complex plane,   we write $B\ominus A$ to denote the set $\{b-a\,:\,b\in B,\,a\in A\}$. ({\it Warning:} In the literature, the   Minkowski difference is  usually defined differently.)

\item All tensor categories (or $\ot$-categories)
 are to be taken in the sense of \cite[Definition 1.1,p.105]{deligne-milne82}. Abelian tensor categories are also understood in the sense of \cite[Definition 1.15,p.118]{deligne-milne82}. 
 \item For a an affine group scheme $G$ over a noetherian ring $R$, we let $\rep RG$ stand for the category of representations of $G$ on finite $R$-modules.  
\end{enumerate}

\subsection*{Acknowledgements}
The research of Ph\`ung H\^o Hai is funded by the Vietnam National Foundation for Science and Technology Development under grant number 101.04-2019.315. He is also partially funded by Vingroup Joint Stock Company and supported by Vingroup Innovation Foundation (VinIF) under the project code VINIF.2021.DA00030.

Both authors thank the referees for various pertinent comments and for pointing out the works of Fiorot, Monteiro Fernandes, Nitsure and Sabbah to us. Thanks are also due to H. Esnault for encouraging an email exchange between some of the aforementioned mathematicians and ourselves. 

\section{Setup}\label{29.01.2020--2}

\subsection{Analytic geometry}In this section we discuss some elementary properties of analytic spaces and sheaves on them for the lack of a suitable reference. We have designed it to reassure the reader with an education in algebraic geometry but with little experience in complex analysis.

Let $f:X\to S$ be a smooth morphism of complex  spaces \cite[Expos\'e 9]{sc}, \cite[Expos\'e 13]{sc}. For any $p\in X$, there exists 
an open   neighbourhood  $U$ and functions  
$x_1,\ldots,x_n\in\Ga(U,\co_X)$ vanishing on $p$  such that the associated morphism \cite[Expos\'e 10, 1.1]{sc}
\[
(x_1,\ldots,x_n,f):U\aro \CC^n\ti S
\]
defines an \emph{isomorphism} of $U$ with $D\ti V$, where $D$ is an open polydisk about $0$  and $V$ an open neighbourhood of $f(p)$. The couple $(U,  x_1,\ldots,x_n)=(U,\bm x)$ is called a \emph{relative system of coordinates about $p$} in this text.
Given such a relative system, we shall find useful to let $\{\pd_{x_j}\}$ stand for the basis of  $\mm{Der}_S(\co_X)|_U=\ch\!om(\Om_{X/S}^1,\co_X)|_U$ dual to $\{  dx_j\}$. We also put $\vt_{x_j}=x_j\pd_{x_j}$.

The following definition is central (and standard): 
\begin{dfn}\label{02.10.2019--1}Let $Y\subset X$ be a closed complex subspace defined by the ideal $\ci\subset\co_X$ \cite[Expos\'e 9, 2.2]{sc} and $p$ a point of $Y$.

\begin{enumerate}[(1)]
\item A relative system of coordinates  $(U,\bm x)$ about $p$ is called {\it adapted} to $Y$ if   $\ci|_U=x_1\cdots x_m\co_U$ for a certain $m\le n$.  
\item We say that $Y$ is a {\it relative effective divisor with normal crossings} if for each $p\in Y$, it is possible to find a relative system of coordinates about $p$ which is adapted to $Y$. 
\item The number $m$ appearing above is called the {\it number of branches}   of $Y$     in $(U,\bm x)$. 
\end{enumerate} 
\end{dfn}

To lighten the text, we will {\it abandon} the adjective ``effective'' in ``effective relative divisor with normal crossings'' in what follows.  
%It should be noted that our notion of ``relative divisor with normal crossings'' is a particular case of the notion of ``positive relative divisor'', which one finds in EGA $\mm{IV}_4$, 21.15. 

Let us now fix a relative divisor with normal crossings $Y\subset X$ and write $\ci$ for its ideal; as customary, $\co_X(Y)$ denotes the invertible $\co_X$-module $\ci^{-1}$ \ega{IV}{4}{21.2.8,p.260} and $\co_X(kY)$ is $\co_X(Y)^{\ot k}$.

%\begin{dfn}\label{05.11.2019--1}For each $q\in Y$, we call the number of  minimal prime divisors of $\ci_q$ \cite[p.39]{mat} the {\it multiplicity}   of $Y$ at $q$. \end{dfn}

Let us now discuss  {\it meromorphic functions} with poles on $Y$. We let $\co_X(*Y)$ stand for the sheaf of $\co_X$-algebras   $\co_X[\ci^{-1}]$ described in \ega{IV}{4}{20.1.1, 226ff}; it is clearly a coherent sheaf of rings on $X$ which   shall be called the sheaf of {\it meromorphic functions with poles on $Y$}. Note that, if $\ci_p=h_p\co_{X,p}$, then the natural morphism 
\[
\co_{X}(*Y)_p\arou\sim\co_{X,p}[1/h_p]
\]
is an isomorphism \ega{IV}{4}{20.1.1.1, p.226}. Alternatively,  $\co_X(*Y)$ can be defined as  the $\co_X$--algebra   $\lid_k\co_X(kY)$. 

As suggested in   \cite{malgrange},  coherent $\co_X(*Y)$-modules shall also be called \emph{$Y$-meromorphic coherent modules on $X$}; the natural base-change functor $\coh{\co_X}\to\coh{\co_X(*Y)}$ is denoted by $\cf\mapsto\cf(*Y)$.

A fundamental property of $Y$-meromorphic coherent modules follows easily from R\"uckert's Nullstellensatz \cite[Ch. 3, \S2, p. 67]{grauert-remmert84}:

\begin{lem}[{\cite[2.3, p.154]{malgrange}}]\label{15.10.2019--1}
If $\cm$ is a $Y$-meromorphic coherent  module supported at $Y$, then $\cm=0$. \qed
\end{lem}
%\begin{proof} The proof is a consequence of R\"uckert's Nullstellensatz \cite[Ch. 3, \S2, p. 67]{grauert-remmert84}. We wish to show that $M_p=0$ for each $p\in Y$. Let $U$ be an open neighbourhood of $p$ where we can find $h\in\co(U)$ such that   $\ci_Y|_U=h\co_U$. Shrinking $U$ if necessary, we can find  $\tilde M\in\coh{\co_U}$ such that  $M|_U=\tilde M\ot_{\co_U}\co_U[1/h]$. Now the support of $\tilde M$ is contained in $Y\cap U$ so that by the Nullstellensatz, $h^t\tilde M=0$ for some $t\in\NN$. This implies that $M|_U=0$.  \end{proof}

\begin{cor}\label{16.06.2016--2}Let $E$ and $F$ be coherent $\co_X(*Y)$--modules. If $X^*=X\smallsetminus Y$, then  the restriction arrow
\[
\hh{\co_X(*Y)}{E}{F}\aro  \hh{\co_{X^*}}{E|_{X^*}}{F|_{X^*}}
\] 
is injective. 
\end{cor}
\begin{proof}Let $\ph:E\to F$ vanish on $X^*$. Working with the coherent $\co_X(*Y)$--module $\mm{Im}(\ph)$, we only need to use Lemma \ref{15.10.2019--1}.  
\end{proof}

%\begin{cor}If $M\in\coh{\co_X(*Y)}$ is such that $M|_{X^*}$ is reflexive, then $M$ is reflexive (as a $\co_X(*Y)$-module, of course). \end{cor} \begin{proof}We consider the canonical arrow $\ph:E\to E^{\vee\vee}$ and note that $\ph|_{X^*}=0$. \end{proof}

\subsection{Connections}\label{25.09.2019--1}

Let $f:X\to S$ be a smooth morphism between analytic spaces. We shall introduce several notations concerning categories of connections. 

\begin{dfn}We let $\g C(X/S)$ stand for the category of integrable $S$-connections $(E,\na:E\to E\ot\Om_{X/S}^1)$, see \cite[1.0-1]{katz70} or \cite[I.2.22]{deligne70}, such that $E$ is a  \emph{coherent} $\co_X$-module. 
\end{dfn} 

For the sake of brevity, in what follows, the word {\it connection} is   synonymous with {\it integrable connection}. With this convention,  $\g C(X/S)$ is the category of $S$-connections. 

Let now $Y\subset X$ be a relative divisor with normal crossings (cf. Definition \ref{02.10.2019--1}); write $\ci$ for the ideal defining $Y$ and $X^*$ for the open subspace $X\smallsetminus Y$.

\begin{dfn}\label{29.01.2020--1}We let  $\g C^{\log}_Y(X/S)$ stand for the category of integrable logarithmic $S$-connections $(E,\na)$, see  \cite[4.0-2]{katz70}, where they are called integrable $S$-connections on $E$ with logarithmic singularities along $Y$, or \cite[I.2.22, 14ff]{deligne70},  such that  the underlying $\co_X$-module $E$ is coherent. 
\end{dfn} 
As in the case of connections, we shall drop the adjective {\it integrable} in ``integrable logarithmic connections.'' 

Note that   $\g C^{\log}_Y(X/S)$ is abelian and has a tensor product rendering the forgetful functor $\g C_Y^{\log}(X/S)\to\coh{\co_X}$ exact,   faithful and tensorial. In addition, for $\ce\in\g C^{\log}_Y(X/S)$, the coherent $\co_X$-module $\ch\!om_X(\ce,\co_X)$ also carries a natural logarithmic connection, see \cite[4.3]{katz70}. 

If $\cd\!er_Y({X/S})$ is the $\co_X$-submodule of $\cd\!er({X/S})$ formed by the derivations which stabilize the ideal $\ci$---it is actually a sheaf of $\co_S$-Lie algebras---then a logarithmic connection on $E\in\coh{\co_X}$ amounts to a morphism of $\co_X$-modules $\na:\cd\!er_Y({X/S})\to\ce\!nd_{\co_S}(E)$ which is compatible with brackets and satisfies $\na(\pd):ae\mapsto \pd(a)e+a\na(\pd)(e)$. See also \cite[4.2]{katz70}.

\begin{ex}\label{19.12.2019--5}The ideal   $\ci=\co_X(-Y)$ carries a tautological logarithmic connection turning it into a  subobject of $\co_X$. Along these lines, each one of the line bundles $\co_X(kY)$ carries a logarithmic connection. 

In general, given $\cf\in\g C_Y^{\log}(X/S)$, we let $\cf(kY)$ stand for the tensor product $\cf\ot_{\co_X}\co_X(kY)$ in $\g C_Y^{\log}(X/S)$. 
\end{ex}

Needless to say, each  $(E,\na)\in\g {C}_Y^{\log}(X/S)$ produces, by restriction to $X^*$,  an $S$-linear connection. Reversing this process is    important for the present work, so we make the:
\begin{dfn}\label{04.11.2019--7}Let $(E,\na)\in \g {C}(X^*/S)$. Any $(\wt E,\wt\na)\in\g C^{\log}_Y(X/S)$ which, when restricted to $X^*$, is isomorphic to $(E,\na)$ shall be called a {\it logarithmic model} of $(E,\na)$. 
\end{dfn}
It should be observed that in the above definition, no assumption is made a priori on the nature of the coherent model. Some simple adjustments can be made, but these shall be left to Section \ref{17.12.2019--1}. 

%\begin{dfn}We let $\mcrs(X^*/S)$ be the full subcategory of $\mc(X^*/S)$ whose object are isomorphic to some object in the image of $\ga$. We refer to objects in  $\mcrs(X^*/S)$ as {\it regular-singular} $S$-connections on $X^*$. (Note that we make {\it no particular} assumption on the coherent modules defining logarithmic models.)  If $(E,\na_E)\in \g C^{\rm rs}(X^*/S)$, any $(\ce,\na_\ce)\in\g C^{\rm log}_Y(X/S)$ such that $\ga(\ce,\na_\ce)\simeq(E,\na_E)$ is called a {\it logarithmic model} of $(E,\na_E)$. \end{dfn}

Because of the sheaf of $Y$--meromorphic functions on $X$, we have yet another relevant category of connections in sight. 

\begin{dfn}\label{04.11.2019--8}
Let $\ce$ be a $Y$-meromorphic coherent module. 
An {\it integrable connection} on $\ce$ is a morphism of $\co_S$-modules 
\[
\na:\ce\aro\ce\otu{\co_X(*Y)}\Om_{X/S}^1(*Y)
\]
having vanishing curvature.  The couple $(\ce,\na)$ is called an \emph{integrable $Y$-meromorphic connection}. 
Morphisms are simply morphisms of $\co_X(*Y)$-modules which respect the connections. The category of integrable $Y$-meromorphic connections is denoted by $\g{MC}_Y(X/S)$.
\end{dfn} 
As above, we shall drop the adjective {\it integrable} in ``integrable $Y$-meromorphic connection.''  

Note also that the base-change functor $\coh{\co_X}\to\coh{\co_X(*Y)}$ defines a functor $\g C_Y^{\log}(X/S)\to\g{MC}_Y(X/S)$ and this prompts us to consider yet another relevant category of connections:
\begin{dfn}\label{04.11.2019--9}Let $(\ce,\na)\in\g{MC}_Y(X/S)$ be given. We say that $(\ce,\na)$ has {\it at worst logarithmic poles} if, for each $p\in X$, there exists an open neighbourhood $U$ of $p$ such that $(\ce,\na)|_U\in\g {MC}_{Y\cap U}(U/S)$ belongs to the image of $\g C^{\log}_{Y\cap U}(U/S)$. The full subcategory of $\g{MC}_Y(X/S)$ having as objects those connections having at worst logarithmic poles shall be denoted by $\g {MC}_Y^{\log}(X/S)$.
\end{dfn}

\begin{rmk}If $S$ is the reduced analytic space associated to a point, we shall usually omit reference to it in notation, that is, in this case, we prefer $\g C(X^*)$ to $\g C(X^*/S)$, etc. 
\end{rmk}

To avoid confusions---the algebraic geometer's reflexes tend to downplay the role of $\g{MC}_Y^{\log}(X/S)$---the reader is asked to bear in mind the following diagram of categories: 
\[
\xymatrix{\g C^{\log}_Y(X/S )\ar[d]^{\text{not nec. full}}&   \\ \g{MC}_Y^{\log}(X/S)\ar[dd]\ar[dr]^{\text{fully faithful}}\\ & \g{MC}_Y(X/S)\ar[ld]^{\text{ faithful}}\\  \g C(X^*/S)& }
\]

\begin{ex}Let $S$ be the point and $X$ be the open unit disc in the complex plane with coordinate function $x$. Let $Y$ be the origin and $\cl=\co_X e$ be a free $\co_X$-module. Define  $
\na_0(e)=0$,  $\na_1( e)= e\ot x^{-1}\mm dx$. Then both $(\cl,\na_0)$ and  $(\cl,\na_1)$ are objects of $ {\g C}_Y^{\log}(X)$. Putting $\na_2(e)=e\ot x^{-2}\mm dx$, we define the object   $(\cl(*Y),\na_2)$  of $\g{MC}_Y(X)$. Note that the images of $(\cl,\na_0)$ and $(\cl,\na_1)$ in $\g{MC}_Y(X)$ are isomorphic, but $(\cl(*Y),\na_2)$ is not isomorphic to any each one of the latter. On the other hand,  the images of $(\cl(*Y) ,\na_2)$ and $(\cl,\na_0)$ in  $\g C(X^*)$ are isomorphic. 
\end{ex}

Let us end this section by highlighting a theorem of Deligne which will prove useful further below. In it, we need the notion of a relative local system \cite[I.2.22, p.14]{deligne70}. 

\begin{dfn}\label{19.12.2019--1}A sheaf of $f^{-1}(\co_S)$-modules $\mathds E$ on $X$ is called a relative local system if for each $p\in X$, there exist open neighbourhoods $U$ of $p$ and $V$ of $f(p)$ with $f(U)\subset V$,  and $M\in\coh{\co_S}$ such that $\mathds E|_U=f^{-1}(M)$.
Morphism between relative local systems are  simply  $f^{-1}(\co_S)$-linear arrows, and the category constructed from this information is denoted by   $\bb{LS}(X/S)$.  
\end{dfn}

 \begin{thm}[{\cite[I.2.23, 14ff]{deligne70}}]\label{10.06.2016--2}Let 
$f:X\to S$ be a smooth morphism of complex  spaces and $(E,\na)$ an object of $\g C(X/S)$.  Then the $\co_S$-module   $\mm{Ker}\,\na$ is a relative local system. In addition, the functor
\[
\g C(X/S)\aro \bb{LS}(X/S), \quad(E,\na)\longmapsto\mm{Ker}\,\na
\]
is an equivalence of categories. 
 \end{thm}

%%%%%%%%%%%%%%%%%%%%%%%%%%%%%%%%%%%%%%%%%%%%%%%%%%%%%%%%%%%%
\section{Extending arrows between meromorphic connections with at worst logarithmic poles}  \label{29.01.2020--3}

 We let $S$ denote the analytic spectrum of $\La$ and give ourselves a smooth morphism   $f:X\to S$  of complex  spaces. 
Let $Y\subset X$ be a divisor with relative normal crossings in $X$ (Definition \ref{02.10.2019--1}) and denote, as usual, $X\smallsetminus Y$ by $X^*$. 

A founding principle of Fuch's theory of regular singularities is that ``solutions are meromorphic''; in this section we wish to show how to adapt Deligne's version of this principle (see \cite[II.4.1, p.86]{deligne70} and \cite[pp. 168-9]{malgrange}) to our present setting. Section \ref{03.05.2016--1} develops the artifice allowing us the adaptation while Section \ref{04.11.2019--3} states and proves the main result (Theorem \ref{21.10.2019--1}). The latter result shall be amplified in Section \ref{04.11.2019--4} after the introduction of residues and exponents in Section \ref{18.12.2019--1}.

%%%%%%%%%%%%%%%
\subsection{Local description of categories of sheaves and connections}\label{03.05.2016--1}
We begin by some categorical remarks.  Let $\cC$ be a $\CC$-linear category. A couple $(c,\al)$ consisting of an object $c\in\cC$ and a morphism of $\CC$-algebras $\al:\La\to \mm{End}_{\cC}(c)$ is called a $\La$-module of $\cC$. Morphisms of $\La$-modules are defined in an obvious way and the category of $\La$-modules in $\cC$ is denoted by $\cC_{(\La)}$.

\begin{ex}Let $A$ be an associative   $\CC$-algebra and let  $\cC$ be the category of finitely generated $A$-modules (on the left say). Then, $\cC_{(\La)}$ is just the category of finitely generated $A\ot_\CC\La$-modules. If $(M,\ca)$ is a topological space with a sheaf of   $\CC$-algebras and $\cC$ is the category of finitely generated $\ca$-modules, then $\cC_{(\La)}$ is just the category of finitely generated $\ca\ot_\CC\La_M$-modules. We notice that in the discussion here, we use repeatedly the assumption that $\Lambda$ has finite dimension over $\CC$. 
\end{ex}

Let 
$D$ be a domain in $\CC^n$ and $H\subset D$ an effective divisor. Suppose that $X=D\ti S$ and $Y=H\ti S$; write  $\mm{pr}:X\to D$ for the canonical projection. 
It follows that $|X|=|D|$, $\co_X=\co_{D}\ot_\CC\La_X$, and $\co_X(*Y)=\co_{D}(*H)\ot_\CC\La_X$.
Denote by $\mm{pr}_*$  any one of the natural functors 
\[
\coh{\co_X}\aro\coh{\co_D} \quad\text{or}\quad
\coh{\co_X(*Y)} \aro\coh{\co_D(*H)}.
\] 
Then $\mm{pr}_*$ induces an equivalence between the categories on the left and the respective categories of $\La$-modules.

Let $\na:\ce\to \ce\ot_{\co_X}\Om_{X/S}^1$ be a connection. Since the natural morphism
\[
 \mm{pr}^*(
\Om_{D}^1 ) \aro \Om_{X/S}^1
\] 
is an isomorphism \cite[Expos\'e 14, 2.1]{sc}, 
the projection formula guarantees that the natural morphism 
\[
\mm{pr}_*(\ce)\ot_{\co_{D}}\Om_{D}^1\aro \mm{pr}_*\left(\ce\ot_{\co_X}\Om_{X/S}^1\right)
\]
is an isomorphism. Furthermore, it is easy to see that the composition 
\[\xymatrix{
\mm{pr}_*\ce\ar[r]^-{\mm{pr}_*(\na)} & \mm{pr}_*(\ce\ot\Om_{X/S}^1) \ar[r]^\sim& \mm{pr}_*\ce\ot_{\co_{D}}\Om_{D}^1
}\] 
is also a connection. From these considerations we in fact   obtain equivalences 
\[
\om:\g C(X/S)\aro \g C(D)_{(\La)}\quad\text{and}\quad \om:\g{MC}_Y(X/S)\aro \left[\g{MC}_{H}(D)\right]_{(\La)}.
\]

We end this section by fixing a definition which will be employed further below.

\begin{dfn}\label{24.06.2016--1}We say that $\ce\in(\coh{\co_D})_{(\La)}$  is a $\La$-vector bundle if each $p\in D$ has an open neighbourhood $U$ such that   $\ce|_U\simeq \co_U\ot_\CC E$ for some finite $\La$-module $E$  and,  for each $\la\in\La$, we have $\la\po(1\ot e)=1\ot\la e$. Analogously, $\ce\in\coh{\co_X}$ is a relative vector bundle, or a vector bundle relatively to $S$, if $\mm{pr}_*(\ce)\in(\coh{\co_D})_{(\La)}$ is a $\La$-vector bundle. 
\end{dfn}

The verification of the following is trivial but shall be employed many times.  
\begin{lem}\label{16.06.2016--1}If $\ce\in(\coh{\co_D})_{(\La)}$ is a $\La$-vector bundle, then $\mm{pr}_*\ce$ is a vector bundle. \qed
\end{lem}

Note that $\ce\in\coh{\co_X}$ is a relative vector bundle if and only if for each $p\in X$  there exists an open neighbourhood $U$ of $p$ such that $\ce|_U$ is the pull-back of a coherent sheaf on $S$. This means that 
Definition \ref{24.06.2016--1}  possesses a natural global version; let us drop  the assumption that $X=D\ti S$. 

\begin{dfn}\label{24.06.2016--1b}
A coherent  $\co_X$-module $\ce$  is a \emph{relative vector bundle}, or a vector bundle relatively to $\La$, if,  locally, it comes from a coherent sheaf on $S$.
\end{dfn}

%\begin{dfn}Let $\ce$ be a relative vector bundle on $X/S$. Given an $S$-valued point $s:S\to X$, we define the type of $\ce$ at $s$ as being the isomorphism class of the $\La$-module $\Ga(S,s^*\ce)$.\end{dfn} 
%Note that for any given point $s:S\to X$, there exists an open $U\subset X$ such that for all other $S$-points $s':S\to X$ such that $s'(|S|)\in U$, the type of $\ce$ at $s'$ is the same as that of $s$. We deduce: 
%\begin{lem}If $|X|$ is connected, then the type of $\ce$ is the same for each $s:S\to X$.\qed \end{lem}

%\begin{cor}\label{26.01.2017--1}A relative vector bundle $\ce$ on $X/S$ is a vector bundle if and only if its type is flat as a $\La$-module. In particular, $\ce$ is a vector bundle if and only if for one point $p\in X$, $\ce_p$ is a  free $\co_{X,p}$-module. \qed \end{cor}

%%%%%%%%%%%%%%%%%%% 
\subsection{Extending morphisms between  connections with at worst logarithmic poles}\label{04.11.2019--3}
We begin by recalling the following result, see  \cite[pp. 163-4]{malgrange}, specially Lemma 5.7. 

\begin{thm}\label{04.11.2019--5}Let $D\subset\CC^n$ be an open polydisk about the origin and denote by $x_j\in\co(D)$ the $j$th coordinate function. Let $H$ be the effective divisor defined by $x_1\cdots x_m\co_D$ and denote by $D^*$ the open subset $D\smallsetminus H$. Let us give ourselves   
  $(E,\na)$ and $(E',\na')$   objects of 
$\g{MC}^{\log}_H(D)$. Then any horizontal arrow $\ph:E|_{D^*}\to E'|_{D^*}$ extends uniquely to an arrow $\wt\ph:E\to E'$. \qed
\end{thm}

Employing the technique of Section \ref{03.05.2016--1}, we can prove the following. 

\begin{thm}\label{21.10.2019--1}The restriction functor 
\[
\g{MC}_Y^{\log}(X/S)\aro\g C(X^*/S)
\]
is fully faithful.  
\end{thm}
\begin{proof}
From Lemma \ref{16.06.2016--2},  the restriction $\coh{\co_X(*Y)}\to\coh{\co_{X^*}}$
is faithful, so that 
$\g{MC}_Y^{\rm log}(X/S)\to\g C(X^*/S)$
is also faithful. 
Let $(E,\na)$ and $(E',\na')$ be in $\g{MC}_Y^{\rm log}(X/S)$ and let $\ph:E|_{X^*}\to E'|_{X^*}$ be  horizontal. We wish to find a horizontal arrow of coherent $\co_X(*Y)$-modules $\wt\ph:E\to E'$ whose restriction to $X^*$ 
is $\ph$. 

Due to uniqueness, it is enough to deal with the  problem locally:  we suppose that $X=D\ti S$ where $D\subset\CC^n$ is an open polydisk about the origin and that 
$Y=H\ti S$, where, writing $\{x_j\}$ for the coordinate functions on $D$, $H$ is defined by $\co_Dx_1\ldots x_m$.
This being the case, we have a commutative diagram whose horizontal arrows are equivalences, 
\[
\xymatrix{
\g{MC}_{Y }(X/S)\ar[d]\ar[rr]^-{\om}   && 
\g{MC}_H(D)_{(\La)}\ar[d]
\\
\g C(X^*/S)\ar[rr]_-{\om}  && \g C(D^*)_{(\La)},
}
\]
as explained on section \ref{03.05.2016--1}. 
We now consider the arrow $\om\ph:\om (E)|_{D^*}\to \om( E') |_{D^*}$ of $\g C(D^*)_{(\La)}$; by Theorem \ref{04.11.2019--5}, $\om\ph$ can be extended to an arrow $\wt{\om\ph}:\om (E)\to\om (E')$ of $\g {MC}_{H}(D)$.  
Because the restriction of $\wt{\om\ph}$ to $D^*$ is a morphism of $\La$-modules, Lemma \ref{16.06.2016--2} assures that $\wt{\om\ph}$ is a morphism of $\La$-modules.  Hence, there exists  $\wt\ph:E\to E'$  in $\g {MC}_{Y}(X/S)$ such that $\om(\wt\ph)=\wt{\om\ph}$. It is clear that $\wt\ph$ is an extension of $\ph$. 
\end{proof}

\section{Residues and exponents of   logarithmic connections}\label{18.12.2019--1}

We work with a smooth morphism of complex  spaces $f:X\to S$ and a  relative divisor with normal crossings $Y$ defined by $\ci\subset\co_X$. We suppose from the start that for each $p\in Y$, the local ring  $\co_{S,f(p)}$   has {\it only one associated prime}.  Finally, we also require

\begin{hyp}\label{14.10.2019--3}There exist a finite set $C$ and a family of smooth and connected relative effective divisors  $\{Y_c\}_{c\in C}$   indexed by $C$ such that $Y=\sum Y_c$. 
\end{hyp}

%%%%%%%%%%%%%%%%%%%%%
\subsection{Residue morphisms}\label{10.06.2016--1}

The objective of this section, Definition \ref{19.12.2019--2}, will need some subsidiary material to  identify the number of branches (Definition \ref{02.10.2019--1}) in terms of local algebra.    This is the reason of Lemma \ref{23.05.2016--1}, its Corollary and Lemma \ref{28.10.2019--1}, which are collected at the end. 

Put  \[Y_c^\dagger:=Y_c\smallsetminus\bigcup_{c'\not=c}Y_{c'} \] and note that 
      $Y\to S$ is  smooth on the points of $Y_c^\dagger$.

Let  
$(\ce,\na)\in\g C_Y^{\log}(X/S)$ be given and, for a fixed $c\in C$, pick  $p\in Y_c^\dagger$. 
Lemma \ref{28.10.2019--1}   enables us to find a relative  system of coordinates adapted to $p$, call it $(U,\bm x)$ (Definition \ref{02.10.2019--1}), such that $\ci|_U=x_1\co_U$. (In particular, $Y\cap U=Y_c\cap U$.) Write $\vt_j$ instead of $\vt_{x_j}$.   
A simple calculation 
 shows that $\na_{\vt_1}$ preserves the $\co_U$-submodule $x_1\ce|_U$ and that the induced morphism of   abelian sheaves  
\[
\mm{Res}_c^{\bm x}:\ce|_{U\cap Y_c}  \aro\ce|_{U\cap Y_c}
\]
is in fact   $\co_{Y_c}$-linear. 
Moreover,  the additive morphisms 
\[
\na_{\vartheta_j}:\ce|_U\aro\ce|_U,\qquad j>1,
\] 
also preserve $x_1\ce|_U$ and hence induce morphisms of abelian sheaves 
\[
\na_{\vartheta_j}:\ce|_{U\cap Y_c}\aro \ce|_{U\cap Y_c}. 
\]
In this way, we define on the $\co_{U\cap Y_c}$--module $\ce|_{U\cap Y_c}$ an $S$-connection $\na_c^{\bm x}$. It is then a simple task to show that $\mm{Res}_c^{\bm x}$ is horizontal for the $S$-connection $\na^{\bm x}_c$.

Now, let $(U',\bm x')$ be  a relative system of coordinates adapted to $Y$ about a point $p'\in U$ and such that $x_1'\co_{U'}=\ci|_{U'}$. (Again $U'\cap Y=U'\cap Y_c$.) 
Since   $x_1\co_{U\cap U'}=x_1'\co_{U\cap U'}$,  we conclude that $x_1'=v_1x_1$ for a certain $v_1\in\co(U\cap U')^\ti$. From this, 
\[\begin{split}
x_1\fr{\pd }{\pd x_1}&=x_1\fr{\pd x_1'}{\pd x_1}\po\fr{\pd}{\pd x_1'}+x_1\sum_{j>1}\fr{\pd x_j'}{\pd x_1}\po\fr{\pd}{\pd x_j'}\\
&=x_1'\fr{\pd}{\pd x_1'}+\fr{x_1'}{v^2_1}\fr{\pd v_1}{\pd x_1}\po x_1'\fr{\pd}{\pd x_1'}+\fr{x_1'}{v_1}\sum_{j>1}\cdots
\end{split}\]
and hence 
\[
\mm{Res}_c^{\bm x}=\mm{Res}_c^{\bm x'}
\]
on $Y_c^\dagger\cap U\cap U'$.  Thus, the $\co_{Y_c^\dagger}$--linear endomorphism 
\[
\mm{Res}_c:\ce|_{Y_c^\dagger}\aro\ce|_{Y_c^\dagger}
\]
is independent of the choice of a relative system of coordinates adapted to $Y$ at $p$.
%with the following property. If $p\in Y_c^\dagger$ and $(U,\bm x)$ is a relative system of coordinates adapted to $Y$ at $p$ such that  $\ci|_U=x_1\co_U$, then $\mm{Res}_c$ is induced by $\na_{\vt_{x_1}}:\ce\to\ce$ upon passage to the quotient.%\qed  
%\end{prp}

\begin{dfn}\label{19.12.2019--2}The $\co_{Y_c^\dagger}$-linear endomorphism $\mm{Res}_c$ constructed 
%in Proposition \ref{19.12.2019--3} 
above is called the residue endomorphism of $\ce$ along $Y_c^\dagger$. 
\end{dfn}

 We end by establishing the side-results employed in giving Definition \ref{19.12.2019--2}. 

\begin{lem}\label{23.05.2016--1}
Let $h:A\to B$ be a  morphism of local noetherian rings and let $\{b_1,\ldots,b_m\}$ be a set  of  elements of $B$ none of which divides zero or is invertible. Suppose that 
\begin{enumerate}[(1)]
\item The ring $A$ has only one    associated prime  $\g r$. 
\item For each $i$, the $A$-module $B/(b_i)$ is flat. 
\item For any $i$, the extended ideal  $\g p_i=(\g r,b_i)$ is prime. 
%\item For distinct $i$ and $j$, the ideals $(b_i)$ and $(b_j)$ are also distinct.
\end{enumerate}
Then $\g p_1,\ldots,\g p_m$ are the minimal prime divisors of $(b_1\cdots b_m)$ and $(b_1),\ldots,(b_m)$ are the corresponding primary components. 
\end{lem}
\begin{proof} To ease notation we write $B_i$ instead of $B/(b_i)$. Using   \cite[IV.2.6, Theorem 2]{bourbakiAC} we have   
\[
\begin{split}\mm{Ass}_B(B_i)&=\mm{Ass}_B(B_i/\g rB_i)\\&=\mm{Ass}_B(B/\g p_i)\\&=\{\g p_i\}. 
\end{split}
\]
This allows us to say, using \cite[Theorem 6.6,p.40]{mat}, that 
$(b_i)$ is $\g p_i$-primary.  
From the Hauptidealensatz, we have $\mm{ht}(\g p_i)\le1$. Now, if  $\mm{ht}(\g p_i)=0$, then $\g p_i$ is a minimal prime ideal containing $(0)$, which implies that $\g p_i\in\mm{Ass}(B)$ due to \cite[Theorem 6.5, p.39]{mat}; this is impossible because $b_i$ is not a zero divisor. In conclusion, $\mm{ht}(\g p_i)=1$.  As $b_1\cdots b_m$ is not a zero divisor, it follows that $\g p_i$ is a minimal prime containing  $(b_1\cdots b_m)$. Now, if $\g q\supset(b_1\ldots b_m)$
 is a prime, 
then $\g q\supset(b_{i_0})$ for a certain $i_0$.  Since $\g q\cap A\supset \g r$, we conclude that $\g q\supset \g p_{i_0}$.  Hence $\g p_1,\ldots,\g p_m$ are the minimal prime ideals containing   $(x_1\cdots x_m)$. In particular, they are also the minimal prime divisors \cite[Theorem 6.5, p.39]{mat}. 
\end{proof}

\begin{cor}\label{18.05.2016--1} Let $p\in Y$. Let $(U,\bm x)$ be a relative system of coordinates about  $p$ which is adapted to $Y$; denote the number of  branches of $Y$ in $(U,\bm x)$  by   $m$. Then $(x_{1,p}),\ldots,(x_{m,p})$ are the minimal primary components of $\ci_p$. \qed\end{cor}

For a given $p\in Y$, let  $C_p=\{c\in C\,:\,p\in Y_c\}$. 
  
\begin{lem}\label{28.10.2019--1}
  Let $(U,\bm x)$ be a relative system of coordinates adapted to $Y$ about $p$. Denote by $m$ the number of branches of $Y$ in $(U,\bm x)$.   There exist an open neighbourhood $V\subset U$ of $p$ and   a bijection $\si: C_p\to\{1,\ldots,m\}$ such that $x_{\si(c)}\co_V$ is the ideal of $Y_{c}$ in $V$. (In particular, $\#C_p$ is the number of branches.) 
\end{lem}
\begin{proof}  Let $h_c\in\co_{X,p}$ generate the ideal of $Y_c$ so that $\ci_p$ is generated by $\prod_{c\in C_p}h_c$. 
By Lemma \ref{23.05.2016--1}, the ideals $\{(h_c)\}_{c\in C_p}$ are the minimal primary components of $\ci_p$.  If $(U,\bm x)$ is as in (2),  Corollary \ref{18.05.2016--1} assures that  $\{(h_c)\}_{c\in C_p}=\{(x_{j,p})\}_{1\le j\le m}$. (Note that   $m$ is also the number of minimal primes of $\ci_p$.)
Let  $c_j\in C_p$ be such that $(h_{c_j})=(x_{j,p})$ and suppose that a certain  $c$ of $C_p$ lies outside $\{c_1,\ldots,c_m\}$. Then  $(h_c)\in \{(x_{j,p})\}_{1\le j\le m}$; suppose, for convenience, that $(x_{1,p})=(h_c)$. Hence $x_{1,p}^2$ divides $\prod_j x_{j,p}$ which means that $x_{1,p}$ divides $\prod_{j>1} x_{j,p}$ (the  product over the empty set being understood as the unit). This is impossible. 
Hence $\#C_p=m$ and the bijection $\si$ is obtained easily.   
\end{proof}

\subsection{Eigenvalues of endomorphisms}\label{30.05.2016--2}
In order to proceed with the study of $\mm{Res}_c$, we shall need some preliminary material on relative local systems. 
We gather here simple fact concerning endomorphism of finite modules and local systems.

\begin{lem}\label{10.10.2019--1}Let $E$ be a finite $\La$-module and $\ph:E\to E$ an endomorphism. Let $\om\ph:E\to E$ denote the $\CC$-linear endomorphism associated to $\ph$. 
\begin{enumerate}\item For each $\vr\in\mm{Sp}_{\om\ph}$, the subspace   $\bb G(\om\ph,\vr)$ is a $\La$-submodule. 
\item If $\ov\ph:\ov E\to\ov E$ stands for the $\CC$-linear endomorphism obtained by reducing $\ph$ modulo $\g m$, we have $\mm{Sp}_{\om\ph}=\mm{Sp}_{\ov\ph}$.\qed
\end{enumerate} 
\end{lem}

Motivated by this result, in what follows, given  an endomorphism of a finite $\La$-module $\ph:E\to E$, we put  
\[\begin{split}
\mm{Sp}_{\ph}&=\begin{array}{c}\text{ the spectrum of $\ph$ regarded}\\\text{ as a $\CC$-linear endomorphism} \end{array}
\\
&=\text{ the spectrum of $\ov\ph:E/\g mE\to E/\g mE$.}  
\end{split}
\]
 
Let $T$ be a locally connected (like all complex spaces \cite[9.3, 177ff]{grauert-remmert84}) and connected topological space and  $\mathds E$ a coherent $\La_T$-module.  (Recall that $\La_T$ is the simple sheaf associated to $\La$: it is in addition a coherent sheaf of rings.) 
Since $\mathds E$ is of finite presentation, each point of $T$ possesses a connected neighbourhood $U$ such that, writing $E=\mathds E(U)$, we have  $\mathds E|_U=E_U$; this is to say that $\mathds E$ is locally constant. In particular, for each connected $U\subset T$ and each $p\in T$, the natural morphism $\mathds E(U)\to\mathds E_p$ is {\it injective}.

Let $\ph:  \mathds E\to \mathds E$ be a $\La_T$--linear endomorphism. It is easily proved that
for any two $p,q\in T$, we have $\mm{Sp}_{\ph_p}=\mm{Sp}_{\ph_q}$; we therefore make the following: 

\begin{dfn}\label{29.10.2019--2} The   {\it spectrum of} $\ph:\mathds E\to\mathds E$ is the spectrum of the  $\La$-linear endomorphism induced by $\ph$ on   an unspecified stalk. It shall be denoted by $\mm{Sp}_\ph$.
\end{dfn} 

Given $\vr\in\mm{Sp}_\ph$, let us write $\bb G(\ph,\vr)$ to denote the pre-sheaf of $\La_T$-modules
\[
U\longmapsto \bigcup_{\nu=1}^\infty \{e\in \mathds E(U)\,:\,(\ph-\vr)^\nu(e)=0\}.
\]
Using that $\mathds E$ is locally  constant, Jordan's decomposition allows us to  say: 

\begin{lem}\label{06.11.2019--1}For each $\vr\in\mm{Sp}_\ph$, there exists $\mu\in\NN$ such that $\bb G(\ph,\vr)=\mm{Ker}(\ph-\vr)^{\mu}$. In particular,   $\bb G(\ph,\vr)$ is a coherent sheaf of $\La_T$--modules  and 
\[
\bigoplus_{\vr\in\mm {Sp}_\ph} \bb G(\ph,\vr) = \mathds E.
\]\qed
\end{lem}

%Let us now assume that $\La$ is moreover a local ring so that  the maximal ideal $\g m$ of $\La$ is nilpotent. 

%\begin{lem}\label{10.10.2019--1}Let $\overline{ \mathds E}=\mathds E/\g m\mathds E$ be considered as a coherent sheaf of $\CC_M$-modules and write $\ov\ph:\ov{\mathds E}\to\ov{\mathds E}$ be the associated endomorphism of $\CC_M$-modules. Then the eigenvalues of $\ov\ph$ and those of $\ph$ coincide.  \end{lem}\begin{proof}This is simple. \end{proof}

\subsection{The exponents}\label{exponents}

Let us now suppose that $S$ is the   analytic spectrum of $\La$ and that the relative divisor $Y\subset X$  satisfies Hypothesis \ref{14.10.2019--3}, whose notations are now in force. Given $(\ce,\na)\in\g C_Y^{\log}(X/S)$, we set out to elaborate on the properties of the residue morphism  (Definition \ref{19.12.2019--2}). 
 
Let $c\in C$ and   $p\in Y_c^\dagger$. We pick a relative system of   coordinates $(U,\bm x)$ adapted to $Y$ at $p$ such that $\ci|_U=x_1\co_U$, see Lemma \ref{28.10.2019--1}. (In particular $Y_c^\dagger\cap U=Y\cap U$.) Then,  
\[
\mm{Res}_c^{\bm x}: \ce|_{U\cap Y_c}\aro\ce|_{U\cap Y_c}
\]
induces an endomorphism 
\[  {\rm HRes}_{c}^{\bm x} : 
\mm{Ker}(\na^{\bm x}_c)\aro\mm{Ker}(\na_c^{\bm x}).
\]
Since $ \mm{Ker}(\na^{\bm x}_c)$ is a local system on $U\cap Y_c^\dagger$ (Theorem \ref{10.06.2016--2})  and  $|U\cap Y_c^\dagger|$ is homeomorphic to a polydisk in $\CC^{n-1}$, we are able to introduce the spectrum of $\mm{HRes}_c^{\bm x}$
as in Definition \ref{29.10.2019--2}. 
Let us deepen our analysis.    By axiom, for any given $q\in Y\cap U$, the set  $\mm{Sp}({\mm{HRes}_c^{\bm x}})$ is just the set of eigenvalues for the $\La$-linear endomorphism
\[
\xymatrix{  \mm{Ker}(\na_c^{\bm x})_q\ar[rr]^-{\mm{HRes}_c^{\bm x}} &&  \mm{Ker}(\na_c^{\bm x})_q}. 
\]  
Now, by Theorem \ref{10.06.2016--2},   the natural arrow  
\[
\mm{Ker}(\na_c^{\bm x})_q\otu\La \co_{X,q}\aro \ce_q
\]
is an isomorphism.  Since  $\g M_{q}\cap\La=\g m$, using a standard isomorphism \cite[II.3.6, Corollary 3]{bourbakiA},
we arrive at a commutative diagram  
\[\xymatrix{
\mm{Ker}(\na_c^{\bm x})_q/\g m\ar[rr]^{\sim}\ar[d]_{ \mm{HRes}_c^{\bm x}  } && \ce_q/\g M_q\ar[d]^{\mm{Res}_c^{\bm x}(q)} 
\\
\mm{Ker}(\na_c^{\bm x})_q/\g m\ar[rr]^{\sim} && \ce_q/\g M_q .
}
\]
 Lemma \ref{10.10.2019--1}-(2) then tells us that $\mm{Sp}({\mm{HRes}_c^{\bm x}})$ is simply the spectrum of the $\CC$-linear endomorphism (independent of $(U,\bm x)$)
\[
\mm{Res}_c (q) : \ce(q)\aro\ce(q). 
\]

\begin{lem}\label{19.12.2019--4} The topological space $|Y^\dagger_c|$ is connected. 
\end{lem}

\begin{proof}
By assumption, $|Y_c|$ is connected. In addition, $|Y_c|$ is the topological space of a complex manifold, call it $Z_c$.  Now, $|Y_c^\dagger|$ is the complement of a thin subset of $|Z_c|$, and hence is connected \cite[7.1.3, p.133]{grauert-remmert84}. 
\end{proof}

Using the fact that $|Y_c^\dagger|$ is a {\it connected} topological space, the following definitions carry no ambiguity. 
\begin{dfn}[The exponents]\label{29.10.2019--3}The spectrum of $\mm{Res}_c(p): \ce(p)\to\ce(p)$ for any point $p\in Y_c^\dagger$ is called the set of exponents of $\na$ along $Y_c$. It shall be denoted by $\mm{Exp}_{Y_c}(\na)$ or $\mm{Exp}_{Y_c}(\ce)$ if no confusion is likely. 
In like manner, $\mm{Exp}_Y(\na)$, or $\mm{Exp}_Y(\ce)$, is the union $\bigcup_c\mm{Exp}_{Y_c}(\na)$.
\end{dfn}

\begin{ex}Let us suppose that $\La=\CC[t]/(t^2)$, that  $X=\mm{Specan}\,\La[x]$     and that $Y$ is defined by $x\co_X$.  Let $\cl=\co_Xe$ be free and put $ \na e= e\ot tx^{-1}\mm dx$. Then $C$ is a singleton  and $\mm{Exp}_Y(\na)=\{0\}$.
\end{ex}

\section{Extending morphisms and connections}\label{04.11.2019--4}

 In this section we set out to prove:

\begin{thm}\label{11.05.2016--2}Let $S$ be the analytic spectrum of $\La$, $f:X\to S$ a smooth morphism of complex  spaces, and $Y\subset X$ a relative divisor with normal crossings satisfying Hypothesis \ref{14.10.2019--3}; write as usual $X^*=X\smallsetminus Y$. 
\begin{enumerate}[(1)]
\item The restriction functor 
\[
\g C^{\log}_Y(X/S)\aro \g C(X^*/S)
\]
is essentially surjective. 
\item The restriction functor 
\[
\g {MC}_Y^{\log}(X/S)\aro \g C(X^*/S)
\]
is an equivalence.
\end{enumerate} 
\end{thm}

Note that, once statement (1) is shown to be true, the work required for proving (2) reduces to the verification of fully faithfulness of $\g {MC}^{\log}_Y(X/S)
\to\g C(X^*/S)$, which is Theorem \ref{21.10.2019--1}. 
Hence, we concentrate on finding preferred extensions for objects of $\g C(X^*/S)$ and aim at Theorem \ref{deligne_manin_extension}.

The idea employed here is the same as in \cite[Theorem 4.4]{malgrange}: one solves the problem locally (this shall be done by Proposition \ref{02.10.2019--2}-(1)) and one proves  then uniqueness of solutions (Proposition \ref{02.10.2019--2}-(2)). The work around local existence follows without much effort from \cite[pp. 159-60]{malgrange}, except that we need to replace the well--known argument concerning the surjectivity of the exponential (Lemma 4.5 in op.cit.) by a slightly longer one (Proposition  \ref{03.09.2019--1}). On the other hand, in extending morphisms, we have chosen to deviate slightly from the standard technique (see  \cite[pp. 161-2]{malgrange} or, for example, \cite[Theorem 4.1]{wasow65}) and put forward Proposition \ref{24.10.2019--2} which shall play a role in Section \ref{17.12.2019--1} as well.

\subsection{The dissonance}
As in the theory of regular-singular connections one finds consistently the need to exclude the case where exponents differ by integers, the following definition shall be useful. 
\begin{dfn}\label{dissonance}
\begin{enumerate}[(1)]\item
Let $A=\{A_j\}_{j\in J}$ and $B=\{B_j\}_{j\in J}$ be   two  families of subsets of $\CC$ indexed by a common   finite set $J$. Define the {\it dissonance} from  $B$ to $A$, call it $\de(B,A)$, as being the maximum of 
\[   \bigcup_{j\in J}\NN\cap(B_j\ominus A_j) 
%\left\{\ell\in \NN\,:\,(\ell+\mm{Exp}_{H_j}(\ce))\cap\mm{Exp}_{H_j}(\cf)\not=\varnothing\right\} 
\]
in case this set is non-empty, and zero otherwise.
 \item Let  $X\to S$ be a smooth morphism,  where $S$ is the analytic spectrum of $\La$. Let $Y=\sum_{c\in C}Y_c$ be a relative divisor with normal crossings  satisfying the assumptions of Hypothesis \ref{14.10.2019--3}. We define, once given $(\ce,\na_\ce)$ and $(\cf,\na_\cf)$ objects of $\g C_Y^{\log}(X/S)$, the {\it dissonance from $\na_\cf$ to $\na_\ce$}, denoted  $\de(\na_\cf,\na_\ce)$, as  the  dissonance from the family $\{\mm{Exp}_{Y_c}(\na_\cf)\}_{c\in C}$ to $\{\mm{Exp}_{Y_c}(\na_\ce)\}_{c\in C}$. 
\end{enumerate}
\end{dfn}
 
\subsection{Extension of arrows: local case}\label{25.09.2019--2}
Let $D$ be an open polydisk in $\CC^n$ about the origin, $\{x_j\}_{j=1}^n$ be the coordinate functions and $H$, respectively $H_j$,  be the effective divisor in $D$ defined by the ideal  $x_1\cdots x_m\co_D$, respectively  $x_j\co_D$. 
Write 
\[\begin{split}
D^*&=D\smallsetminus H,\\  H^\dagger_j&=H_j\smallsetminus \bigcup_{k\not=j}^mH_k
\end{split}\]
and 
\[\begin{split}  H^{\rm cross}&=\bigcup_{i<j} H_i\cap H_j\\& = H\smallsetminus  \bigcup_jH_j^\dagger. \end{split}\]

\begin{prp}\label{24.10.2019--2}
Let $(\ce,\na_\ce)$  and $(\cf,\na_\cf)$   be objects of $\g C^{\log}_H(D)$, $\de=\de(\na_\cf,\na_\ce)$ the dissonance from $\na_\cf$ to $\na_\ce$ and    $\ph:\ce(*H)\to \cf(*H)$  an arrow in $\g {MC}_H^{\log}(D)$. If $\cf$ is a vector bundle, then   $\mm{Im}(\ph)\subset \cf(\de H)$.
More precisely, there exists an arrow 
\[\wt\ph:\ce\aro\cf(\de H)\]
in $\g C_H^{\log}(D)$
rendering the diagram 
\[
\xymatrix{\ce(*H)\ar[rr]^\ph && \cf(*H)
\\
\ce\ar[u]^{\rm can}\ar[rr]^{\wt\ph}&& \cf(\de H)\ar[u]_{\rm can}
}
\] 
commutative.  In addition, $\wt \ph$ is the unique arrow of $\mm{Hom}_{\co_{D}}(\ce,\cf(\de H))$ extending $\ph$. 
 
\end{prp}

\begin{proof}We are only required to show that $\ph$ extends to an arrow $\wt\ph:\ce\to\cf(\de H)$ between $\co_D$-modules; indeed, since for each $U\subset D$ the natural arrow $\Ga(U,\cf (\de H))\to\Ga( U,\cf( * H))$ is  injective, the fact that $\ph$ is horizontal immediately assures that $\wt\ph$ is likewise. 

The heart of the proof is the 
 \emph{Claim:} For any $j\in\{1,\ldots,m\}$, any $p\in  H^\dagger_j$,  any open polydisk $V\subset D\smallsetminus H^{\rm cross}$ about $p$, and any $e\in \Ga(V,\ce)$, the section $\ph(e)\in\Ga(V,\cf(*H))$ is actually an element of  $\Ga(V,\cf(\de H))$.

To ease notations, we prove the claim in case $j=1$. 
Let $p\in H_1^\dagger$ and fix $V$ an  open polydisk about  $p$ and contained in $D\smallsetminus H^{\rm cross}$ (in particular, $V\cap H_1=V\cap H_1^\dagger$). Let $U\subset V$ be a polydisk about $p$ where we can find a  relative system of coordinates adapted to $H$, call it $\bm y$, and  such that $y_1\co_U$ is the ideal of $H_1\cap U$.

In what follows, let us write $\vt$ in place of   $\vt_{y_1}$, \[\vt_\ce:\ce|_U\aro\ce|_U\quad\text{and}\quad\vt_\cf:\cf|_U\aro\cf|_U\] for the $\CC$--linear endomorphisms  defined by $\vt$ by means of the connection. 
We will require the following  simple result: 

\begin{lem}\label{23.10.2019--1} For any $k\in\NN$,  $\mu\in\NN$ and $\vr\in\CC$, the following formula holds:  
\[
[\vt-(\vr+k)]^\mu y_1^k=y_1^k(\vt-\vr)^\mu.  
\]\qed
\end{lem}

Adopting the notations of Section \ref{exponents} and employing Lemma \ref{06.11.2019--1}, we can write  
\[
\mm{Ker}(\na^{\bm y}_{\ce,1})  = \bigoplus_{\vr} \mm{Ker}\left(\mm{Res}_1(\ce) -\vr\right)^{\mu_\vr}.
\]
(In Section \ref{exponents} we have actually employed the notation $\mm{HRes}$ for the restriction of $\mm{Res}$ to the sheaf $\mm{Ker}$; this chance should cause no confusion.)
Let $\vr\in\mm{Exp}_{H_1}(\ce)$ and choose  $\ov e\in \Ga(U\cap H_1 ,\ce|_{H_1})$  such that 
\[
\ov e \in\mm{Ker} \left(\mm{Res}_1(\ce)-\vr\right)^{\mu_\vr}.
\]  
Because $U$ is a polydisk (and hence a Stein space), there exists      $e\in\Ga(U,\ce)$ such that      $e|_{H_1}=\ov e$ \cite[1.4.6, p. 35]{grauert-remmert84}. Consequently,
\[
(\vt_\ce-\vr)^{\mu_\vr}(e)\in\Ga(U,y_1\ce)=y_1\Ga(U,\ce).
\]
Let $k>\de$ be   such that $y^k_1\ph(\ce)\subset \cf$.  
 %Let $\vr\in\mm{Exp}_{H_1}(\ce)$, $q\in  H_1$ and   $s\in\ce_q\smallsetminus \g m_q\ce_q$ be such that $(\na_\ce-\vr)^\mu(s)\in \g m_q\ce_q$. (That such an element exists, follows from the definition of the residue.)  
By Lemma \ref{23.10.2019--1}, we have 
\[
\begin{split}
[\vt_\cf-(\vr+k)]^{\mu_\vr}(y_1^k\ph(e))&  = y_1^k(\vt_\cf-\vr)^{\mu_\vr}(\ph(e)) 
\\
&=y_1^k\ph[(\vt_\ce-\vr)^{\mu_\vr}(e)]
\\&\in y_1^{k+1}\ph(\ce).
\end{split}
\]
Hence,   
\[
[\vt_\cf-(\vr+k)]^{\mu_\vr}(y_1^k\ph(e))\in\Ga(U,y_1\cf).
\]
This implies that 
\[
\left.\left\{y_1^k\ph(e)\right\}\right|_{H_1} \in \mm{Ker}\,(\mm{Res}_1(\cf)-(\vr+k))^{\mu_\vr}.
\]
Since $\vr+k$ cannot be an exponent of $\cf$, since $k>\de$ is supposed true, $\left.\left\{y_1^k\ph(e)\right\}\right|_{H_1}$ has to vanish altogether which means that 
\[
y_1^k\ph(e)\in  \Ga(U,y_1\cf).
\] (Here we used the fact that the only   section $s$   of a vector bundle   $\cs$ over a {\it reduced} complex space $Z$ such that $s(z)=0$ for all $z\in Z$ is the zero section.)
Hence, \[ y_1^{k-1}\ph(e)\in\Ga(U,\cf)\] since $\cf$ has no $y_1$-torsion. 

Let now $e\in \Ga(U,\ce)$ be an arbitrary section and write 
\[\tag{$\#$}
e|_{H_1} = \sum_\vr \ov a_\vr\po \ov e_\vr,
\]
where $\ov a_\vr\in\co(U\cap H_1)$ and  $\ov e_\vr\in\mm{Ker}(\mm{Res}_1(\na)-\vr)^{\mu_\vr}$. This is possible because 
\[
\co_{U\cap H_1}\ot\mm{Ker}(\na_{\ce,1}^{\bm y})\arou\sim \ce|_{U\cap H_1}.
\]
By the Stein property, we may assume that $\ov e_\vr=e_\vr|_{H_1}$  and $\ov a_\vr= a_\vr|_{H_1}$ so that, using eq. ($\#$)  and multiplying by $y_1^{k-1}$, we have 
\[
y_1^{k-1}e=\sum_\vr a_\vr \po y_1^{k-1} e_\vr+y_1^k \ep,\quad\text{for some $\ep\in\Ga(U,\ce)$.}
\]
Hence, 
$\ph(y_1^{k-1}e)\in \Ga(U,\cf)$. 
By induction, we are then able to say that $\ph(y_1^{\de}e)\in \Ga(U,\cf)$. It is now a simple matter to show that the Claim holds in all generality.

Granted the claim, we can assure that for any $p\in H$, any open polydisk $V$ about $p$ in $D$ and any $e\in\Ga(V,\ce)$, the section $\ph(e)|_{V\smallsetminus H^{\rm cross}}$ belongs to $\Ga(V\smallsetminus H^{\rm cross},\cf(\de H))$. Using  Riemann's second extension theorem \cite[p.132, 7.1.2]{grauert-remmert84} we are   able to say that $\ph(e)\in\Ga(V,\cf(\de H))$. It is then a simple matter to see that $\mm{Im}(\ph)\subset\cf(\de H)$.

We end by arguing that the arrow constructed previously is unique. Let $\wh\ph:\ce\to\cf(\de H)$ extend $\ph$. It then follows that $\wh\ph-\wt\ph$ when restricted to $D^*$,  coincides with $0:\ce|_{D^*}\to\cf|_{D^*}$. But a section of $\cf(\de H)$ over an open subset $W\subset D$ which vanishes on $W\smallsetminus H$ must vanish on the whole of $W$ by the identity theorem \cite[I.4.10, p.22]{fritzsche-grauert02}, so that $\wh\ph=\wt\ph$ as maps from $\Ga(W,\ce)$ to $\Ga(W,\cf(\de H))$; this is enough argument.
\end{proof}

\subsection{Extension of arrows: global case} Due to the uniqueness of the extension, the global case is reduced to the local case, in which we can assume the settings of Section \ref{03.05.2016--1}. 

%The work here has all been done by Corollary \ref{08.10.2019--2}, and we simply record the outcome for future reference. 

\begin{prp}\label{20.11.2019--1}Let $S$ be the  analytic spectrum of $\La$ and 
$f:X\to S$ a smooth morphism. Let $Y\subset X$ be a relative divisor with normal crossings satisfying   Hypothesis \ref{14.10.2019--3} and denote the complement of $Y$ by   $X^*$. Let $(\ce,\na_\ce)$ and $(\cf,\na_\cf)$ be objects of $\g C_Y^{\log}(X/S)$, and suppose that $\cf$ is a vector bundle relatively to $S$. Denote by $\de$ the dissonance   from $\na_\cf$ to $\na_\ce$. 
Then, any arrow $\ph:\ce|_{X^*} \to \cf|_{X^*}$ in $\g{C} (X^*/S)$ can be extended to an arrow $\wt\ph:\ce\to\cf(\de Y)$ of $\g C_Y^{\log}(X/S)$. Moreover, this extension is unique. 
\end{prp} 
\begin{proof}
Due to the uniqueness, the claim is local in $X$. Therefore, employing the notations from the beginning of Section \ref{25.09.2019--2},  we can assume: 
\begin{eqnarray*}
X&=D\ti S,
\\   
Y&=H\ti S &\text{and}
\\ 
Y_j&=H_j\ti S. 
\end{eqnarray*}
In addition, let us employ also the notations from Section  
\ref{03.05.2016--1}. In particular, $\om$ denotes the equivalence  $\g C_Y^{\log}(X/S)\stackrel\sim\to \g C^{\log}(D)_{(\La)}$ explained in the aforementioned section. 
The following result is tautological and will be used further ahead.  
\begin{lem}\label{30.05.2016--1}Let $(\ce,\na_\ce)\in {\g C}_Y^{\log}(X/S)$. Then, for each $j\in\{1,\ldots,m\}$, we have  $\mm{Exp}_{Y_j}(\na_\ce)=\mm{Exp}_{H_j}(\om\na_\ce)$. In particular, if $(\cf,\na_\cf)$ is another object of $ {\g C}_Y^{\log}(X/S)$, then we have an equality of dissonances $\de(\cf,\ce)=\de(\om\cf,\om\ce)$.
\qed
\end{lem} 

%Applying   Theorem \ref{04.11.2019--5} 
%Theorem \ref{08.10.2019--1} 
Continuing with the proof, we have  objects $\om\ce$ and $\om\cf$ from $[\g C_H^{\log}(D)]_{(\La)}$, and an arrow   
\[
\om\ph:\om\ce|_{D^*}\aro \om\cf|_{D^*}
\]
in $\g C(D^*)_{(\La)}$. From Lemma \ref{30.05.2016--1},   $\de(\om\cf,\om\ce)=\de$. 
We are in a position to apply Theorem \ref{04.11.2019--5}  and then, as $\cf$ is locally free, Proposition \ref{24.10.2019--2}
to obtain    
% Theorem \ref{08.10.2019--1} 
 an arrow from $\g C_H^{\log}(D)$,  
\[
\ov{\om\ph} : \om\ce\aro (\om\cf)(\de H) ,
\]
 extending $\om\ph$. 
The identity principle  \cite[I.4.10,p.22]{fritzsche-grauert02}    assures that the restriction morphism
\[\mm{Hom}_{\co_D}(\om\ce,(\om\cf)(\de H))\aro\hh{\co_{D^*}}{\om\ce|_{D^*}}{\om\cf|_{D^*}}\]
is injective and consequently $\ov{\om\ph}$ is also $\La$-linear.   Hence, $\ph$ extends to a horizontal morphism  $\wt\ph:\ce\to \cf(\de Y)$ because we can identify $\om(\cf(\de Y))$ with $(\om\cf)(\de H)$.

Injectivity of the restriction morphism $\hh{\co_X}{\ce}{\cf(\de Y)}\to\hh{\co_{X^*}}{\ce|_{X^*}}{\cf|_{X^*}}$ guarantees that $\wt\ph$ is unique. 
\end{proof}
%\qed
\subsection{Extension of connections: local case} In this section, notations and conventions are  those described at the start of Section \ref{25.09.2019--2}: $D$ is an open polydisk, etc.

\begin{thm}[``Local Deligne-Manin extension'']\label{02.10.2019--2}Let $(E,\na )$ be an object of $\g C(D^*)_{(\La)}$.  
\begin{enumerate}[(1)]
\item   
There exists $\ce\in {\g C}_H^{\log}(D)_{(\La)}$ and an isomorphism in $\g C(D^*)_{(\La)}$,  
\[
\ph:  \ce|_{D^*}\arou\sim E. 
\]
In addition, $\ce$ can be chosen to enjoy the following properties. 
\begin{enumerate}[(i)]
\item As an $\co_D$-module with action of $\La$, $\ce$ is
a $\La$-vector bundle (see Definition \ref{24.06.2016--1}).   
\item For each $j\in\{1,\ldots,m\}$, the exponents of $\ce$ along $H_j$  are all on $\tau$.
\end{enumerate}
\item  Let  $\ce'\in {\g C}_H^{\log}(D)_{(\La)}$ enjoy properties (i) and (ii) of (1), and let 
\[
\ps: \ce' |_{D^*}\arou\sim E
\]
be an isomorphism. Then there exists an isomorphism 
\[
\xi:\ce\arou\sim  \ce'
\]
in $ {\g C}^{\log}_H(D)_{(\La)}$ rendering diagram 
\[
\xymatrix{\ce|_{D^*}\ar[r]^-{\ph}\ar[dr]_{\xi|_{D^*}} &E  \\ & \ce' |_{D^*}\ar[u]_\ps }
\]
commutative. 
Moreover, the isomorphism $\xi$ is the only one having this property. 
\end{enumerate}
\end{thm}
 
\begin{proof}(1) Let $b=(b_{1},\ldots,b_{n})\in D^*$   and write $\Ga$ for the fundamental group of $D^*$ based at $b$. Let 
\[\ga_1(t)=(b_{1}e^{2\pi \mm it},b_{ 2},\ldots ), \ga_2(t)=(b_{1},b_{2}e^{2\pi \mm it},\ldots),\quad\text{etc}\] so that $\Ga$  is the free abelian group generated by $\{\ga_j\}$. 
The dictionary of \cite[I.2, 5 ff.]{deligne70}  produces an equivalence of categories 
\[
\g C(D^*)_{(\La)}\arou\sim\mm{Rep}_\CC(\Ga)_{(\La)},\quad F\longmapsto F(p_0).\] Let $g_1,\ldots,g_m\in \mm{Aut}_\CC(E(p_0))$ be associated to $\ga_1,\ldots,\ga_m$; obviously, each $g_j$ is $\La$-linear. Employing Proposition \ref{03.09.2019--1}, we can find  $T_1,\ldots,T_m\in\mm{End}_\La (E(p_0))$ such that 
\begin{itemize}\item  $
\exp(-2\pi \mm i T_j)=g_j$,  
\item the eigenvalues of every $T_j$ belong to $\tau$.
\item   $T_j$ commutes with the remaining $T_k$. 

\end{itemize}
Let $\ce=\co_D\ot_\CC E(p_0)$ be endowed with the obvious action of $\La$ and define 
\[
\na_{\ce} (1\ot e)  =\sum_{j=1}^m [1\ot T_j(e)]\ot\fr{dx_j}{x_j}\in\Ga\left(D\,,\,\ce\ot\Om^1_{D}(\log H)\right).\]  Since $T_jT_k=T_kT_j$, the   connection $\na_\ce$ is integrable. 
Clearly, $\La$ acts by horizontal endomorphisms because each $T_j$ is $\La$-linear. This shows that $(\ce,\na_{\ce})$ is an object of $ {\g C}^{\log}_H(D)_{(\La)}$ enjoying property  $(i)$. 
Obviously, the residue endomorphism 
\[
\ce|_{H_j} \aro \ce|_{H_j} 
\]
sends a section $1\ot e$ to $1\ot T_j(e)$, so that condition $(ii)$ is fulfilled by construction of $T_j$. Finally, the restriction of $(\ce, \na_\ce)$ to $D^*$ is an object of $\mc(D^*)_{(\La)}$ which corresponds, under the equivalence $\g C(D^*)_{(\La)}\simeq \mathrm{Rep}_\CC(\Ga)_{(\La)}$, to the representations defined by 
$\ga_j\mapsto g_j$ because $\exp(-2\pi\mm iT_j)=g_j$. 
Consequently, the restriction of $\ce$ to $D^*$ is isomorphic to $E$.

(2)
Let now $\ce'$  and $\ps$ be as in  (2). 
By Proposition \ref{24.10.2019--2}
%Theorem \ref{08.10.2019--1},   
there exists an arrow in  $ {\g C}^{\log}_H(D)$, call it $\xi: \ce  \to   \ce'$, 
such that $\xi|_{D^*}=\ps^{-1}\ph$.
As the restriction of $\xi$ to $\g C(D^*)$ commutes with the actions of $\La$ and the restriction map 
\[\hh{\co_D}{\ce}{\ce'}\aro\hh{\co_{D^*}}{\ce|_{D^*}}{\ce'|_{D^*}}\]
is injective,  we conclude that $\xi$ commutes with the action of $\La$ as well. Arguing with $\ph^{-1}\ps$ instead of $\ps^{-1}\ph$, it is not hard to see that $\xi$ is an isomorphism from $\g C^{\log}_H(D)_{(\La)}$. The injectivity of the restriction arrow again proves the uniqueness statement concerning $\xi$. 
\end{proof}

\subsection{Extension of connections: the global case} Let $S$ be the analytic spectrum of  $\La$ and 
$f:X\to S$ a smooth morphism. Let $Y\subset X$ be a relative divisor with relative  normal crossings satisfying in addition Hypothesis \ref{14.10.2019--3}. As usual, we write $X^*$ for $X\smallsetminus Y$.  
 Putting Theorem \ref{02.10.2019--2},  the equivalences of Section 
\ref{03.05.2016--1} and Lemma \ref{30.05.2016--1} together,   we obtain: 
\begin{thm}[Deligne-Manin extensions]\label{deligne_manin_extension}Let $(E,\na)$ be an arbitrary object of $\mc(X^*/S)$. The following claims are true. 

\begin{enumerate}[(1)]
\item There exists $ \ce \in {\g C}_Y^{\log}(X/S)$ and an isomorphism 
\[
\ph:\ce|_{X^*}\arou\sim E  
\]
in $\g C(X^*/S)$. Moreover,  $ \ce$ can be found to have  the following properties: 
\begin{enumerate}[(i)]\item The $\co_X$-module $\ce$ is a   vector bundle relatively to $S$.  

\item The exponents of $ \ce $ along $Y$ are all on $\tau$. \end{enumerate}

\item Let $ \ce'\in {\g C}^{\log}_Y(X/S)$ enjoy properties (i) and (ii) of $(1)$, and let $\ps:\ce'|_{X^*}\stackrel \sim \to  E$
be an isomorphism. Then there exists an isomorphism 
$\xi: \ce\stackrel \sim \to \ce'$
rendering the diagram 
\[
\xymatrix{    \ce |_{X^*} \ar[r]^-{\ph} \ar[dr]_-{\xi|_{X^*}} & E   \\  &\ce' |_{X^*}\ar[u]_\ps           }
\]
commutative. Moreover, the isomorphism $\xi$ is the only one having this property.  \end{enumerate}
\end{thm}

Given Theorem \ref{deligne_manin_extension}, the proof of point (1) in Theorem \ref{11.05.2016--2} is concluded. Also, as argued after its statement, the verification of item (2) is also finished.

 \subsection{An exercise on   the matrix exponential }\label{13.11.2019--1}Let   $M$ be a finite $\La$-module. In    proving   Theorem \ref{02.10.2019--2}  we needed the following simple result.

\begin{prp}\label{03.09.2019--1}
Let $\al,g_1,\ldots,g_m\in \mm{Aut}_\La(M)$ be given and suppose that $\al$ commutes with each $g_j$. Then, there exists a unique $X\in\mm{End}_\La(M)$ such that $\exp(2\pi \mm iX)=\al$ and $\mm{Sp}_X\subset\tau$. In addition, $X$ commutes with each $g_j$. 
\end{prp}
\begin{proof}
Let $Z_j$ stand for the center of $g_j$ in the linear algebraic group $\mm{Aut}_\La(M)$ and let   $\al=s\po u$ be the Jordan-Chevalley decomposition of $\al$ in $\cap_jZ_j$ \cite[2.4, 29ff]{steinberg74}.   Since $u-\mm {Id}$ is nilpotent, the series   
\[N = \fr1{2\pi\mm i}\sum_{k\ge1}\fr{(-1)^{k-1}}k (u-\mm {Id})^k\]
is in fact a sum and   $\exp(2\pi \mm i N)  = u$. Obviously $N$ belongs to $\mm{End}_\La (M)$ and commutes with each $g_j$. 

For every $\vr\in\mm{Sp}_{\al}=\mm{Sp}_{s}$, let $\ell(\vr)$ stand for the unique element of $z\in\tau$ such that $e^{2\pi\mm i z}=\vr$. 
Because $s$ is semi-simple, we have the direct sum decomposition  
\[
M=\bigoplus_{\vr\in\mm{Sp}_s} \bb E(s,\vr),
\] 
where in addition, each $\bb E(s,\vr)$ is stable under $u$, $\La$ and each $g_j$. Define $S\in\mm{End}_\CC(M)$ by decreeing that 
\[
S|_{\bb E(s,\vr)} = \ell(\vr)\po\mm{Id}.
\]
It then follows that $S\in \mm{End}_\La (M)$, that $S$ commutes with $u$, and a fortiori with $N$, and that $\exp (2\pi\mm iS)= s$. In addition, $S$ commutes with each $g_j$. 
 We now put $X=S+N$ which shows the existence of an element in $\mm{End}_\La (M)$ commuting with each $g_j$, whose spectrum is contained in $\tau$ and such that  $\exp(2\pi\mm iX)=\al$. The following well-known Lemma establishes uniqueness. \end{proof}

\begin{lem}Let $X$ and $X'$ be $\CC$-linear endomorphism of $M$ such  $\mm{Sp}_{X}$ and $\mm{Sp}_{X'}$ are contained in $\tau$. Then $\exp (2\pi\mm iX)=\exp(2\pi\mm iX')$ implies $X=X'$.   
\end{lem}
\begin{proof} This is a well-known exercise, but we were unable to find a reference.  Let us sketch a proof. One first deals with the case where $X$ and $X'$ are diagonalisable, resp. nilpotent. Then one applies the additive Jordan-Chevalley decomposition together with the fact that the exponential sends diagonalisable endomorphism, resp. nilpotent,  to diagonalisable automorphisms, resp. unipotent. 
%Since the function $z\mapsto e^{2\pi iz}$ from $\tau$ to $\CC$ is injective,  we conclude that $\mm{Sp}_X=\mm{Sp}_{X'}$. Let us  suppose that $X$ and $X'$ are diagonalizable; in this case,  it is easy to see that, for each $\vr\in\mm{Sp}_X=\mm{Sp}_{X'}$,  \[ \bb E(\exp(2\pi iX),e^{2\pi i\vr})=\bb E( X,\vr)\quad\text{and}\quad \bb E(\exp(2\pi iX'),e^{2\pi i\vr})=\bb E(X',\vr). \] Hence,  \[\bb E(X,\vr)=\bb E(X',\vr), \] which shows that $X=X'$. 
\end{proof}

\section{Connections in the algebraic and analytic case}\label{17.12.2019--1}

The following conventions are fixed in this section: $S$ is a noetherian $\CC$-scheme,   $f:X\to S$ is  a smooth morphism of   $\CC$-schemes, $Y\subset X$ a relative effective (positive) divisor having relative normal crossings \cite[4.0, p.187]{katz70}. The ideal of $Y$ shall be denoted by $\ci$. 
Write  $X^*=X\smallsetminus Y$ and let $u:X^*\to X$ stand for the inclusion. 
\subsection{Connections in the algebraic case}\label{31.10.2019--1}

We define the categories $\g C(X/S)$ and  $\g C_Y^{\log}(X/S)$ in complete analogy with  Section \ref{25.09.2019--1} (see also \cite[Section 4]{katz70}). On the other hand,  if we follow the path of Section \ref{25.09.2019--1} and introduce  $\g {MC}_Y(X/S)$, we obtain simply $\g C(X^*/S)$. Indeed,  we see that $\co_X(*Y)=u_*(\co_{X^*})$ and because $u$ is an affine morphism \ega{II}{}{1.2.1, p.6}, we conclude that $\coh{\co_X(*Y)}$ is none other than $\coh{\co_{X^*}}$ \ega{II}{}{1.4, 9ff}.

\begin{dfn}
We let $\g C_Y^{\rm rs}(X^*/S)$ be the {\it full} subcategory of $\g C(X^*/S)$ whose object are isomorphic to some object in the image of the restriction functor $\g C_Y^{\log}(X/S)\to\g C(X^*/S)$. (More succinctly, $\g C_Y^{\rm rs}(X^*/S)$ is the essential image of the restriction.) We refer to objects in  $\g C^{\rm rs}_Y(X^*/S)$ as {\it regular-singular} $S$-connections on $X^*$.  If $(E,\na_E)\in \g C^{\rm rs}_Y(X^*/S)$, any $(\ce,\na_\ce)\in\g C^{\rm log}_Y(X/S)$ such that $ (\ce,\na_\ce)|_{X^*}\simeq(E,\na_E)$ is called a {\it logarithmic model} of $(E,\na_E)$. 
\end{dfn}

\begin{rmk}The reference to $Y$ in the notation $\g C^{\rm rs}_Y(X^*/S)$ is there to remind us of the dependence of $X$; it is envisageable to develop, as in \cite{deligne70}, a more general theory, but we have chosen not to do so. 
\end{rmk}

Note that  {\it no particular} assumption on the coherent modules defining logarithmic models is made. Of course, it is possible to simplify certain models as the following constructions show.  

Let $\ce\in\coh{\co_X}$;  define \[(0:\ci^k)_\ce=\ch\!om_{X}(\co_X/\ci^k,\ce);\]
this is naturally a coherent submodule of $\ce$ supported  on $Y$. Put $(0:\ci^\infty)_\ce=\bigcup_k(0:\ci^k)_\ce$ and note that   $(0:\ci^\infty)_\ce$ is the kernel of $\ce\to u_*(u^*\ce)$. 

\begin{dfn}[{\ega{IV}{2}{5.9.9, p.112}}]We say that $\ce$ is $Y$-pure if $(0:\ci^\infty)_\ce=0$.  
\end{dfn}
We now note that obtaining $Y$-pure logarithmic models is always possible:

\begin{lem}\label{10.01.2020--1}For each   $(\ce,\na)\in\g C_Y^{\rm log}(X/S)$, the $\co_X$-submodule  $(0:\ci^k)_\ce$ is a subconnection. In particular, each $(E,\na)\in\g C_Y^{\rm rs}(X^*/S)$ has a logarithmic model $\ce$ which is $Y$-pure.\qed
\end{lem}

%\begin{proof} Let $x_1,\ldots,x_n$ be etale coordinates on an affine and  open subset $U$ of  $X$. Suppose that $\ci|_U=\co_Uh$, where $h=x_1\cdots x_m$. Let $e\in\ce(U)$ be annihilated by $h^k$. Then \[ \begin{split}0&=\na_{\vt_{x_i}}(he)\\&=k h^k e+h^k\na_{\vt_{x_i}}(e)\\&=h^k\na_{\vt_{x_i}}(e). \end{split} \] This means that $\na_{\vt_{x_i}}(e)$ is still annihilated by $h^k$.  \end{proof}

Let now $E\in\g C^{\rm rs}_Y(X^*/S)$  and let $F\subset E$ be a subobject. If $\ce\in\g C^{\log}_Y(X/S)$ is a $Y$-pure logarithmic model, so that $\ce$ is naturally a $\co_X$-submodule of $u_*(E)$, we define  $\cf:=u_*(F)\cap\ce$, which is a coherent $\co_X$-module \ega{I}{}{9.2.2, p.172}.  
In addition, it is clear that $\cf$ is stable under $\cd\!er_Y(X/S)$, and hence $\cf$ is   a logarithmic model for $F$. In the same vein, if $E\to Q$ is an epimorphism with kernel $N$, let $\cn\subset \ce$ be the logarithmic model of $N$ previously constructed. It then follows that $\cq:=\ce/\cn$ is a logarithmic model for $Q$ and we have verified the truth of the first two statements of the following proposition. The final claims are very simple and  left without proof. 

\begin{prp}\label{16.01.2019--1}The full subcategory $\g C^{\rm rs}_Y(X^*/S)$ of $\g C(X^*/S)$ is stable under subobjects and quotients. It is also stable under tensor products and duals. \qed
\end{prp}

\subsection{Construction of models when the base is    $\spc\La$}\label{31.01.2020--4}
In addition to the assumptions of the beginning of the section, we   suppose that $f$ is {\it proper}, that $S=\spc \La$ and that $Y=\sum_cY_c$, where each $Y_c$ is connected and smooth over $S$. We write $\g X$, $\g S$,   $\g Y$, etc for the analytifications of $X$, $S$, $Y$, etc. Note that $\g Y$ now satisfies Hypothesis \ref{14.10.2019--3} since each $Y_c^\an=\g Y_c$ is connected \cite[XII.2.4]{SGA1} and smooth \cite[XII.3.1]{SGA1}. 
As usual, $X^*=X\smallsetminus Y$ and $\g X^*=\g X\smallsetminus\g Y=(X\smallsetminus Y)^\an$. 

The following two results, Proposition \ref{GAGA} and Lemma \ref{05.11.2019--2}, shall be employed in structuring our arguments. The first one is a simple consequence of GAGA \cite[Expos\'e XIII, Theorem 4.4]{SGA1}  and the unconvinced reader will find some details in \cite[5.4]{hai-dos_santos19}. While Lemma \ref{05.11.2019--2} is essentially GAGA, we find necessary to proceed   carefully. 

\begin{prp}[GAGA]\label{GAGA}
The analytification functor 
\[
\g C_Y^{\log}(X/S)\aro \g C_{\g Y}^{\log}(\g X/\g S)
\]
is an equivalence. 
\qed
\end{prp}

\begin{lem}\label{05.11.2019--2}Let $\ce$ and $\cf$ be coherent $\co_X$-modules and assume that $\cf$ is $Y$-pure. Then 
\[
\hh{\co_X(*Y)}{\ce(*Y)}{\cf(*Y)}\aro \hh{\co_{\g X}(*\g Y)}{\ce^{\an}(*\g Y)}{\cf^\an(*\g Y)}
\]
is bijective. 
\end{lem}
\begin{proof}[Sketch of proof.] We require two steps.   

{\it Step 1.} We claim that the natural arrows
\[\Ph:
 \lid_k\hh{\co_X}{\ce}{\cf(kY)}\aro \hh{\co_X}{\ce}{\cf(*Y) } 
\]
and 
\[\Ps:
 \lid_k\hh{\co_{\g X}}{\ce^\an}{\cf^\an(k\g Y)}\aro \hh{\co_{\g X}}{\ce^\an}{\cf^\an(*\g Y) } 
\]
are bijective. Since the argument is sufficiently general, we give it only in the case of $\Ph$. 

 Even without assuming $\cf$ to be $Y$-pure, quasi-compacity of $X$ shows that $\Ph$ is injective. Surjectivity   requires  $Y$-purity since we need    the natural arrows $\cf(kY)\to \lid_\ell\cf(\ell Y)=\cf(* Y)$ to be injective in order to glue.

{\it Step 2.}
Using the natural bijections 
\[
\hh{\co_X}{\ce}{\cf(*Y)}=\hh{\co_X(*Y)}{\ce(*Y)}{\cf(*Y)}
\]
and 
\[
\hh{\co_{\g X}}{\ce^\an}{\cf^\an(*\g Y)}=\hh{\co_{\g X}(*\g Y)}{\ce^\an(*\g Y)}{\cf^\an(*\g Y)}, 
\] 
the Lemma shall be proved once we establish that the natural arrow 
\[
\hh{\co_X}{\ce}{\cf(*Y)}\aro\hh{\co_{\g X}}{\ce^\an}{\cf^\an(*\g Y)}
\]
is bijective. But this follows from GAGA  \cite[XIII, Theorem 4.4]{SGA1} and bijectivity of $\Ph$ and $\Ps$.  

\end{proof}

\begin{thm}[Deligne-Manin extensions]\label{22.10.2019--1}Let $(E,\na)$ be an  object of $\g C_Y^{\rm rs}(X^*/S)$. Then, there exists  $\wt E\in\g C_Y^{\log}(X/S)$ enjoying the two properties enumerated below and an isomorphism $\Ph:\wt E|_{X^*}\stackrel\sim\to E$. 
 
\begin{enumerate}[(1)]\item All the exponents of $\wt E^\an$ (this is an object of $\g C_{\g Y}^{\log}(\g X/\g S)$) along $\g Y$ lie in $\tau$.
\item   The $\co_{\g X}$-module $\wt E^\an$ is a relative vector bundles. 
\end{enumerate}
In addition, if $\wt F\in\g C_Y^{\log}(X/S)$ also enjoys (1) and (2) above and $\Ps:\wt F|_{X^*} \to E$ is an isomorphism, then  there exists an isomorphism $\Xi:\wt E\to \wt F$ such that $\Ph|_{X^*} = \Ps\circ\Xi|_{X^*}$.
\end{thm}
\begin{proof}Let $\cm\in\g C^{\log}_Y(X/S)$ be a logarithmic model of $E$; we may assume that $\cm$ is $Y$-pure because of Lemma \ref{10.01.2020--1}. Applying  Theorem \ref{deligne_manin_extension}, there exists   $\wt{\g E}\in\g C_{\g Y}^{\log }(\g X/\g S)$ 
enjoying properties $(i)$ and $(ii)$ of the said result, and a horizontal isomorphism 
\[\ph:
\wt{\g E}|_{\g X^*}\arou\sim E^\an. 
\]
Using Theorem \ref{21.10.2019--1}, it is possible to find an arrow 
\[\wt\ph:\wt{\g E}(*\g Y)\aro \cm^\an(*\g Y)  \]
in $\g {MC}_{\g Y}(\g X/\g S)$ extending $\ph$. By GAGA, there exists $\wt E\in\g C^{\rm log}_{Y}(X/S)$ whose analytification is $\wt{\g E}$. By  Lemma \ref{05.11.2019--2}, there exists a   morphism of $\co_X(*Y)$-modules
\[
\wt{\ph}^{\rm alg}:  \wt{E}(*Y)\aro \cm(*Y)
\]
such that  $\left(\wt{\ph}^{\rm alg}\right)^\an=\wt\ph$. 
 In particular, $\left(\wt\ph^{\rm alg}|_{X^*}\right)^\an$ is just $\ph$ and hence is a horizontal isomorphism. Since the analytification functor $\coh{\co_{X^*}}\to\coh{\co_{\g X^*}}$ is conservative   \cite[XII.1.3.1, p.241]{SGA1} (it is exact and faithful),   it follows that $\wt\ph^{\rm alg}|_{X^*}$ is a horizontal isomorphism and $\wt E$ is a logarithmic model for $E$.

The last statement is a direct consequence of claim (2) in Theorem  \ref{deligne_manin_extension} and GAGA. 
\end{proof}

\subsection{Logarithmic models   in the case of a   complete base }\label{31.01.2020--1}
Let   $R$ be a complete noetherian local $\CC$-algebra with residue field $\CC$. If $\g r\subset R$ stands for the maximal ideal, we shall write $R_k$ instead of $R/\g r^{k+1}$: these are all finite dimensional complex vector spaces. 

We now  place ourselves in the situation explained in the beginning of this section and {\it add   other    assumptions}: 
\begin{itemize}\item The scheme  $S$ is $\spc R$,
\item    the morphism $f:X\to S$ is {\it proper with connected fibres},
\item the divisor $Y$ is $\sum_cY_c$,  where for each $c$, the $S$-scheme $Y_c$ is smooth,  and
\item  the $\CC$-scheme   $Y_c\ot R_0$ is   connected. (Recall that in this case, $(Y_c\ot R_0)^\an$ is likewise connected \cite[XII, Proposition 2.4, p.243]{SGA1}.)
\end{itemize}
The reader is directed to Remark \ref{temkin} below to have an idea of how the setting imposed by the above hypothesis comes about for a given $X^*$ over $S$.

To simplify notation,  given $k\in\NN$, we write 
\[S_k=\spc{R_k},\quad X_k=X\ot_RR_k, \quad Y_k=Y\ot_RR_k,\quad\text{etc}.
\] 
These produce, via the analytification functor, complex   spaces $\g S_k$, $\g X_k$,    $\g Y_k$, etc.

We wish to compare $\g C_Y^{\rm rs}(X/S)$ and the category of representations on $R$-modules of the fundamental group of $\g X_0$ in analogy with \cite{hai-dos_santos19} and \cite[7.2.1, p.170]{malgrange} (or \cite[Theorem II.5.9, p.97]{deligne70}). 
We begin with the analogue of \cite[Definition 5.2]{hai-dos_santos19}. 

\begin{dfn}
\label{dfn_28.07.2021--1}The category $\g C_Y^{\rm log}(X/S)^\wedge$ is the category whose 
\begin{enumerate}\item[objects] are sequences $(\ce_k,q_k)$, with   $\ce_k\in\g C_{Y_k}^{\log}(X_k/S_k)$  and $q_k:\ce_{k+1}|_{X_{k+1}}\to \ce|_{X_k}$  an isomorphism in $\g C_{Y_k}^{\log}(X_k/S_k)$, and 
\item[arrows] between $(\ce_k,q_k)$ and $(\cf_k,r_k)$ are compatible families of arrows $\ph_k:\ce_k\to\cf_k$. 
\end{enumerate}
We introduce analogously the categories $\g C(X^*/S)^\wedge$, $\g C(\g X/\g S)^\wedge$ and $\g C(\g X^*/\g S)^\wedge$ (and note that here $\g X$, $\g X^*$ and $\g S$ carry no mathematical meaning). 
\end{dfn}

We begin with a   consequence of Grothendieck's algebraization theorem \cite{illusie05}. Its proof is simple and omitted;  the unconvinced reader might find further details in  \cite[Proposition 5.3]{hai-dos_santos19}.   
\begin{prp}[The GFGA equivalence]\label{GFGA}The natural morphism 
\[
\g C_Y^{\log}(X/S)\aro \g  C_{  Y}^{\log}(  X/  S)^\wedge
\] 
is an equivalence. \qed
\end{prp}

We can now find certain preferred logarithmic models for objects in $\g C^{\rm rs}_Y(X^*/S)$.

\begin{thm}\label{01.10.2019--1}Any $E\in \g C_Y^{\rm rs} (X^*/S)$ possesses a logarithmic model $\ce$   enjoying the following properties. 
\begin{enumerate}[(1)]
\item For each $k\in\NN$, the analytification $(\ce|_{X_k})^\an$, which is an object of $\g C^{\log}_{\g Y_k}(\g X_k/\g S_k)$,  has all its exponents on $\tau$. 
\item For each $k\in\NN$, the coherent sheaf  $(\ce|_{X_k})^\an$ is a vector  bundle relatively to $\g S_k$.  
\end{enumerate} 
\end{thm} 
\begin{proof}We shall write $E_k$ instead of $E|_{X^*_k}$ and let 
 $q_k$ denote   the canonical isomorphism $E_{k+1}|_{X_k^*}\to E_k$. 
Let 
\[
\wt E_k\in\g C_{Y_k}^{\log}(X_k/S_k)
\] 
be the logarithmic model of $E_k$ obtained from an application of Theorem \ref{22.10.2019--1}: $\wt E_k^\an$ has all its exponents on $\tau$, is a vector bundle relatively to $\g S_k$ and there exists an isomorphism $\ph_k:\wt E_k|_{X_k^*}\to E_k$. 

\vspace{.3cm}
\noindent{\it Claim.} For each $k\in\NN$, $q_k$ 
extends to an isomorphism in $\g C^{\log}_{Y_k}(X_k/S_k)$:  
\[
\xymatrix{\wt E_{k+1}|_{X_{k}}\ar[rr]^-{\wt q_k}&& \wt E_{k}.} 
\] 

\noindent{\it Proof.}
Let us consider the morphisms $\g q_k$ and $\g u_k$ defined by the commutative diagrams 
\[
\xymatrix{
\left\{\left.\left(\wt E_{k+1}|_{X_k}\right)\right|_{X_k^*}\right\}^\an \ar[rr]^-{\g q_k}\ar[rr]\ar@{=}[d]
&&
\left\{ \wt E_k|_{X_k^*}\right\}^\an\ar[dd]^\sim
\\
\left\{\left.\left(\wt E_{k+1}|_{X_{k+1}^*}\right)\right|_{X_k^*}\right\}^\an\ar[d]_\sim&& 
\\
\left(E_{k+1}|_{X_k^*}\right)^\an\ar[rr]_{(q_k)^\an}  && E_k^\an   
}
\]
and   
\[
\xymatrix{
\left\{\left.\left(\wt E_{k+1}|_{X_k}\right)\right|_{X_k^*}\right\}^\an \ar[rr]^-{\g q_k}\ar[rr]\ar@{=}[d]
&&
\left\{ \wt E_k|_{X_k^*}\right\}^\an \ar@{=}[d]
\\
\left.(\wt E_{k+1}|_{X_k})^\an\right|_{\g X_k^*} \ar[rr]_-{\g u_k} &&\wt E_k^\an|_{\g X_k^*}.
}
\]
(Note that the first diagram is the image of a diagram from $\g C(X_k^*/S_k)$ by the analytification functor.)
 The fact that   $(\wt E_{k+1}|_{X_k})^\an\simeq\wt E_{k+1}^\an|_{\g X_k}$ and $\wt E_k^\an$ have exponents on $\tau$ and are vector bundles relatively to $\g S_k$ allows us to   apply Theorem \ref{deligne_manin_extension}-(2) and assure that $\g u_k=\wt{ \g u}_k|_{\g X_k^*}$ for a certain isomorphism $\wt {\g u}_k: (\wt E_{k+1}|_{X_k})^\an\to\wt E_k^\an$.
By GAGA (Proposition \ref{GAGA}), $\wt{\g u}_k=(\wt q_k)^\an$ for  some isomorphism $\wt q_k:\wt E_{k+1}|_{X_k}\to \wt E_k$.
By functoriality, we see that $\g q_k=(\wt q_k|_{ X_k^*})^\an$. Since $(-)^\an:\coh{\co_{X_k^*}}\to\coh{\co_{\g X_k^*}}$ is faithful \cite[XII.1.3.1, p.241]{SGA1}, $\wt q_k$ is the desired extension of $q_k$, which completes the proof of the claim.

Using the GFGA equivalence (Proposition \ref{GFGA}) and the family $\{\wt q_k\}$,  we can find   $\wt E\in\g C_Y^{\log}(X/S)$ such that 
\[
\wt E|_{X_k}=\wt E_k,\quad\text{for each $k$.}
\] 

\vspace{.3cm}
\noindent{\it Claim.}  The logarithmic connection $\wt E$ is a model of $E$. 

\noindent{\it Proof.}
This is  not automatic since the natural arrow $\g C(X^*/S) \aro \g C(X^*/S)^\wedge$ does not have to be full.  
Let $\cm\in\g C^{\log}_{Y}(X/S)$ be any logarithmic model of $E$ so that $(\cm|_{X_k})^\an$ is a logarithmic model of $(E|_{X_k})^\an$. Hence, by  Theorem \ref{21.10.2019--1}, there exists a {\it unique} isomorphism 
\[
\xymatrix{ (\cm|_{X_k})^\an(*\g Y_k)\ar[rr]^-{\g g_k}_-{\sim}&& ( \wt E|_{X_k})^\an(*\g Y_k)}
\]
 in $\g {MC}^{\rm log}_{\g Y_k}(\g X_k/\g S_k)$  whose restriction to   $\g X_k^*$ is the identity of $(E|_{X_k})^\an$. 
 Let $\de$ be the dissonance from $(\wt E|_{X_k})^\an$ to $(\cm|_{X_k})^\an$  for an arbitrary $k$. 
By Proposition \ref{20.11.2019--1}, there exists  an unique arrow in $\g C_{\g Y_k}^{\log}(\g X_k/\g S_k)$
\[
\xymatrix{ (\cm|_{X_{k}})^\an \ar[rr]^-{\g h_k} && (\wt E|_{X_k})^\an(\de\g Y_k)}
\]
 extending   $\g g_k$. 
In particular,  $\g h_k|_{\g X_k^*}=\id_{(E|_{X_k})^\an}$, and uniqueness assures that $\g h_{k+1}|_{\g X_k}=\g h_k$.

Let 
\[
\xymatrix{ \cm|_{X_k}\ar[rr]^{h_k} && \wt E(\de Y)|_{X_k}
}
\]
be an arrow of $\g C_{Y_k}^{\log}(X_k/S_k)$
inducing $\g h_k$ after analytification. Since $\g h_{k+1}|_{\g X_k}=\g h_k$, we conclude also that $h_{k+1}|_{X_k}=h_k$. Hence, by GFGA, it is possible to find a morphism   in $\g C_Y^{\log}(X/S)$, 
\[
\xymatrix{
\cm\ar[rr]^{h} && \wt E(\de Y),
}
\]
such that $h|_{X_k}=h_k$ for each $k$. In particular, $(h|_{X^*})|_{X_k^*}$ is an isomorphism for each $k$, which shows that $h|_{X^*}: \cm|_{X^*}\to \wt E|_{X^*}$ is an isomorphism  on an open neighbourhood of the closed fibre   $X_0^*$ \ega{I}{}{10.8.14, p.198}.  From Lemma \ref{24.10.2019--1}, we deduce that $h|_{X^*}$ is an isomorphism so that $\wt E(\de Y)$, and hence $\wt E$, is a logarithmic model of $E$. 
\end{proof}

The following result was employed in proving Theorem \ref{01.10.2019--1} and shall be useful also in the verification of Theorem \ref{21.11.2019--1}. 

\begin{lem}\label{24.10.2019--1}Let $T$ be the spectrum of a  local and noetherian $\CC$-algebra with closed point $o$.  
Let $g:Z\to T$ be a smooth morphism with connected fibres. 
 \begin{enumerate}[(1)]\item
 Let $\cf$ an object of $\g C(Z/T)$ such that  
 $\cf_p=0$ for each   $p\in Z_o$.  Then $\cf=0$.
 \item Let $\ph: \ce\to\ce'$ be an arrow of $\g C(Z/T)$ which is an isomorphism, respectively vanishes,  on an open neighbourhood of $Z_o$. Then $\ph$ is an isomorphism, respectively vanishes allover. 
  \end{enumerate}  
\end{lem}
\begin{proof}(1) We need to show that the open subset $\mm{Supp}(\cf)^{\rm c}$ is in fact $Z$ and start by noting    that $\mm{Supp}(\cf)^{\rm c}$ has the following distinctive property. If $p\in\mm{Supp}(\cf)^{\rm c}$, then the fibre $g^{-1}(g(p))$ is also contained in $\mm{Supp}(\cf)^{\rm c}$. Indeed, if $\cf_p=0$, then $\cf|_{g^{-1}(g(p))}(p)=0$ and, since $\cf|_{g^{-1}(g(p))}$ is locally free \cite[Proposition 8.8,p.206]{katz70}, we conclude that $\cf|_{g^{-1}(g(p))}=0$. Consequently, if   $q\in g^{-1}(g(p))$, then   $\cf(q)=0$ and hence, by Nakayama's Lemma,  $\cf_q=0$. 
Once this has been put under the light, it follows easily that $\mm{Supp}(\cf)^{\rm c}=Z$ because $g(\mm{Supp}(\cf)^{\rm c})$ is  open in $T$ \ega{IV}{2}{2.4.6, p.20} and contains the closed point.

(2) Follows from (1) applied to      $\mm{Ker}(\ph)$ and $\mm{Coker}(\ph)$, respectively to $\mm{Im}(\ph)$. 
\end{proof}

\begin{rmk}\label{temkin}Let $Z^*$ be a smooth and separated scheme over  $S=\spc R$. Following \cite[Theorem 3.2]{lutkebohmert}, it is always possible to find a proper scheme $Z$ and a schematically dense open immersion $Z^*\to Z$, that is, a compactification of $Z^*$. The question is then when one can assume that $Z$ is in addition smooth over $S$ and the boundary $D:=(Z\smallsetminus Z^*)_{\rm red}$ is a divisor with relative normal crossings. This is a much harder question, to which we have no complete answer.    If $R$ is regular, \cite[Theorem 1.1]{temkin} tells us that $Z$ above can be chosen to be  regular as well. In addition, the same theorem assures  that it is possible to find a situation where $D$ has normal crossings (not relative ones). The case in which one looks for smooth morphisms has been recently studied in \cite{atw}, but to communicate their main results would take us   far afield without providing a straightforward answer to the reader. 
\end{rmk}

\subsection{The relative Deligne-Riemann-Hilbert correspondence}
\label{31.01.2020--2} 
Assumptions and notations are the same as on Section \ref{31.01.2020--1}.

\begin{thm}\label{21.11.2019--1}
The natural  functor \[\nu:\g C_Y^{\rm rs}(X^*/S)\aro \g C(\g X^*/\g S)^\wedge\] is an equivalence.  
\end{thm}
\begin{proof}
Consider the natural composition of functor  
\[
 {\g C}^{\log}_Y(X/S) \aro \g C_Y^{\rm rs} (X^*/S)\arou\nu  \g C (\g X^* /\g S )^\wedge ,
\]
call it $\mu$. The result shall be proved once we establish that $\mu$ is essentially surjective and that $\nu$ is fully faithful. 

{\it Essential surjectivity of $\mu$.} This is very similar to the beginning of the proof of Theorem \ref{01.10.2019--1} and we leave to the reader to fill in the details. 

%Let $(E_k;q_k )\in  \g C (\g X^*/\g S)^\wedge$; using Theorem \ref{deligne_manin_extension}, there exists, for each $k$, an a logarithmic model    $\wt E_k\in\g C_{\g Y_k}^{\log}(\g X_k/\g S_k)$ of $E_k$ which is a vector bundle relatively to $\g S_k$ and whose exponents are  on $\tau$.  Consequently $\wt E_{k+1}|_{\g X_{k}}$ is also a vector bundle relatively to $\g S_{k}$ whose exponents belong to $\tau$. Using the isomorphism  \[ \xymatrix{ E_{k+1}|_{\g X_{k}^*} \ar[rr]^{q_k}&&E_{k} } \] and part (2) of Theorem \ref{deligne_manin_extension}, there exists an isomorphism   \[ \xymatrix{\wt E_{k+1}|_{\g X_{k}} \ar[rr]^{\wt q_k}&& \wt E_{k}} \]  of ${\g C}^{\log}_{\g Y_k}(\g X_{k}/\g S_{k})$ rendering commutative the diagram  \[ \xymatrix{\left(\wt E_{k+1}|_{\g X_{k+1}^*}\right)|_{\g X_k^*}\ar[rr]^\sim&&E_{k+1}|_{\g X_{k}^*}   \ar[rr]^-{q_k} && E_{k}\\ &&\ar[u]\ar@{=}[ull]\left(\wt E_{k+1}|_{\g X_{k}}\right)|_{\g X^*_{k}}\ar[rr]_-{\wt q_k|_{\g X_{k}^*}}&&\ar[u] \wt E_{k} |_{\g X_{k}^*}.  }\]
%\newcommand{\alg}{{\rm alg}}
%By GAGA, there is $\wt\ce_k\in\g C_{Y_k}^{\log}(X_k/S_k)$ whose analytification is $\wt E_k$. In addition, there exists an isomorphism 
%\[ \xymatrix{  \wt\ce_{k+1}|_{X_k}       \ar[rr]^-{ \wt q_k^{\alg}}  &&       \wt\ce_k   }\] in $\g C_{Y_k}^{\log}(X_k/S_k)$ whose analytification is $\wt q_k$.  By GFGA (Proposition \ref{GFGFA}), there exists $\wt\ce\in\g C_Y^{\log}(X/S)$ such that \[\wt\ce|_{X_k}=\wt\ce_k;\] this gives an object of $\g C_Y^{\log}(X/S)$ which sent to $\{E_k,q_k\}$ under  $\mu$. 

{\it Fullness of $\nu$.} Let  $E$ and $F$ be in $\g C^{\rm rs}(X/S)$ and consider 
\[
\left\{\g g_k:  (E |_{X_k^*} )^\an
\aro
 (F|_{ X^*_k} )^\an\right\}_{k\in\NN}
\]
a compatible system of arrows defining a morphism of $\g C(\g X /\g S)^\wedge$.   Let $\ce$ and $\cf$ be logarithmic models of $E$ and $F$ satisfying conditions (1) and (2) of Theorem \ref{01.10.2019--1}. Once  $E|_{X_k^*}$, respectively $F|_{X_k^*}$, is identified with $(\ce|_{X_k})|_{X_k^*}$, respectively $(\cf|_{X_k})|_{X_k^*}$, 
Proposition   \ref{20.11.2019--1} shows that each $\g g_k$ can be extended to an arrow 
\[
\wt{\g g}_k: (\ce|_{ X_k})^\an \aro (\cf|_{ X_k})^\an
\]
in $\g C^{\log}_{\g Y_k}(\g X_k/\g S_k)$. Now  GAGA assures that $\wt{\g g}_k$ is the analytification of a horizontal arrow 
\[
\wt g_k  : \ce|_{X_k}\aro\cf|_{X_k}.
\]
Unravelling all identifications and using the uniqueness statement in Proposition  \ref{20.11.2019--1}, we can assure that 
\[
\wt g_{k+1}|_{X_k} =\wt g_k.
\]
By GFGA (Proposition \ref{GFGA}), there exists a  horizontal arrow $\wt g:\ce\to\cf$  such that $\wt g|_{X_k}=\wt g_k$. In particular, if $g=\wt g|_{X^*}$, then $(g|_{X_k^*})^\an=\g g_k$.

{\it Faithfulness of $\nu$.} This follows from faithfulness of each $(-)^\an:\coh{\co_{X_k^*}}\to\coh{\co_{\g X_k^*}}$ \cite[XII, Proposition 1.3.1, p.241]{SGA1},  from \ega{I}{}{10.8.13, p.198} and from part (2) of Lemma \ref{24.10.2019--1}. 
\end{proof}
  
Let now $\xi:S\to X^*$ be a section of $X\to S$ and denote by  $\Ga$ the fundamental group of the topological space  $( X^*_0)^\an$ based at the point $\xi(S_0)$. 

Proceeding as in Section 5.5 and 5.6 of \cite{hai-dos_santos19}, it is not difficult to see that the $R$-linear $\ot$-category  $\g C(\g X^*/\g S)^\wedge$ is equivalent to $\rep R\Ga$. 
  
\begin{cor}
 [Riemann-Hilbert correspondence] 
\label{25.11.2019--1}Let $\xi:S\aro X^*$ and $\Ga$ be as before.  Then the fiber functor
$\xi^*:\g C^{\rm rs}_Y(X^*/S)\aro \modules R$ factors through an     equivalence of $R$-linear $\ot$-categories $\Ph:\g C^{\rm rs}_Y(X^*/S)\stackrel\sim\aro \rep R{\Ga}$. 
%In addition, the composition of $\Ph$ with the forgetful functor $\rep R{\Ga}\to\modules R$ is isomorphic to   $M\mapsto \xi^*M$. 
\end{cor}
\begin{proof}
Theorem \ref{21.11.2019--1} yields a functor
\[\nu:\g C_Y^{\rm rs}(X^*/S)\aro \g C(\g X^*/\g S)^\wedge\]
which is compatible with the fiber functors induced by $\xi$. Recall that the target of $\nu$ 
is the pro-system of categories 
$\g C(\g X^*_k/\g S_k)$ (Definition \ref{dfn_28.07.2021--1}). According to Theorem \ref{10.06.2016--2} we have an equivalence
\[ \g C(\g X^*_k/\g S_k)\stackrel\sim\aro
\bb{LS}(\g X^*_k/\g S_k)\]
which is also compatible with the fiber functor given by $\xi_k=\xi\otimes_RR_k$. Further,     $\xi_k^*$ yields an equivalence
\[ \bb{LS}(\g X^*_k/\g S_k)
\stackrel\sim\aro \rep{R_k}\Gamma.\]
Therefore we obtain an equivalence
\[ \g C(\g X^*/\g S)^\wedge\stackrel\sim\aro
\rep{R_k}\Gamma^\wedge
\stackrel\sim\aro \rep R\Gamma.\]
Composition with $\nu$ now gives us $\Ph$ mentioned in the statement.
%
%$\Phi$ is the composition of functor $\nu$ in Theorem \ref{21.11.2019--1} and a functor $\mu$, which is the limit of the functors:
%$$\mu_k:
%\g C(\g X^*_k/\g S_k)^\wedge\stackrel\sim
%\to \rep{R_k}{\Ga}^\wedge.$$
%Indeed, $\rep R{\Ga}$ is the limit of the category $\rep{R_k}{\Ga}$ in the sense of Definition \ref{dfn_28.07.2021--1}. 
%Now, each $\mu_k$ is the composition of the functor $\om$ defined in Section \ref{03.05.2016--1} and the classical Riemann-Hilbert correspondence.
\end{proof}
\section{Applications to the calculation of differential Galois groups}
\label{31.01.2020--5}

The problem of computing the differential Galois group of a connection (in particular of a system of linear differential equations) on a complex variety is not an easy problem, even for a simple case such as the affine line, see, e.g., \cite{katz87} and references therein.  On the other hand, in case of regular-singular connections, it is possible to employ the complex analytic picture to compute by means of Schlesinger's Theorem: the differential Galois group is the closure of the monodromy group. 

In our relative setting, where the base is the spectrum of a complete discrete valuation ring, thanks to Corollary \ref{25.11.2019--1} and the results obtained in   \cite{hai-dos_santos19, duong-hai-dos_santos18}, we are able to explicitly compute the differential Galois group of connections.  

Let $R$ now be $\pos\CC t$ and  $S$ be its  spectrum. 
We shall work in the setting of Section \ref{31.01.2020--1}, which is repeated here for   convenience. 
As in the previous sections, we abbreviate $R/(t^{k+1})$ to $R_k$ and $\spc R_k$ to $S_k$. 
Let   
 $f:X\to S$ be a proper and smooth morphism with  connected fibres and $Y\subset X$ a relative effective (positive) divisor having normal crossings in the sense of \cite[4.0, p.187]{katz70}. In addition, we assume that   $Y$ is a finite sum $\sum_cY_c$ where, for each $c$,     $Y_c$ is $S$-smooth and $Y_c\ot_RR_0$ is connected. 
As usual, $X^*=X\smallsetminus Y$,  a smooth $R$-scheme having connected fibres. 
Finally,  we write 
\[
X_k=X\ot_RR_k,\quad\text{and} \quad X^*_k=X^*\ot_RR_k.
\]

In this setting we can apply Tannakian duality to define   differential Galois groups at a point (there are two such groups) of a vector bundle on $X^*$ endowed with a relative connection. This shall be reviewed briefly in Section \ref{26.02.2020--1} below. 

%For more details the reader is referred to \cite[Section 7]{duong-hai18}.

%We set out to compute differential Galois groups for connections on $X^*$ using Corollary \ref{25.11.2019--1}, so we begin by a brief review. 

\subsection{The differential Galois groups of a vector bundle with a relative   connection}\label{26.02.2020--1}
 Let us fix an $R$-point $\xi$ of $X^*$ and write $\Ga$ for the topological fundamental group of $(X^*_0)^\an$ based at $\xi(S_0)$.   To avoid repetitions, all group schemes in sight are {\it affine and flat over $R$}.  
 
 Recall that given  $M$ and $M'$ in $\g C(X^*/S)$, we can form the objects $M\op M'$, $M\ot M'$ and $H\!om( M,M')$ as Katz explains in \cite[1.1]{katz70}. If $M$ is a vector bundle, we denote by $\check M$ the object $H\!om( M,\co')$.
  
Let $(M,\na)\in\g C(X^*/S)$ be such that $M$ is a vector bundle and introduce the full subcategory $\langle M\rangle_\ot$  of $\g C(X^*/S)$ consisting of ``{\it subquotients  of objects tensor generated by $(M,\na)$}'', i.e.  the objects of $(M,\na)$ are 
\[
\left\{ N\in\g C(X^*/S)\,:\,\begin{array}{c} 
\text{There exist non-negative integers $\{a_j,b_j\}_{j=1}^r$, a   } \\ \text{subobject $N'$ of $(M^{\ot a_1}\ot \check M^{\ot b_1})\op\cdots\op (M^{\ot a_r}\ot \check M^{\ot b_r})$ }
\\ \text{and a horizontal epimorphism $N'\to N$}\end{array}\right\}\eqno(\dagger).
\] 
Then,  \cite[Thm. II.4.1.1]{saavedra72} (see  pages 1109 and 1140 in   \cite{duong-hai18} as well) says that the functor which associates to $N\in\langle M\rangle_\ot$ its pull-back via $\xi$, 
\[
\xi^*: \langle M\rangle_\ot\aro\modules R, 
\]
induces a $\ot$-equivalence between $\langle M\rangle_\ot$ and the category $\rep R{\mm{Gal}'(M)}$, where  $\mm{Gal}'(M)$ is a group scheme  called the {\it full differential Galois} group of $M$  
  \cite[Section 7]{duong-hai-dos_santos18}.  Contrary to the case of a base field,  $\mm{Gal}'(M)$ might easily fail to be of finite type \cite[3.2.1.5]{andre01}, \cite[Example 7.11]{duong-hai-dos_santos18}. In particular, the natural morphism of group schemes 
\[
\xi^*: \mm{Gal}'(M)\aro {\bf GL}(\xi^*M)
\]
is usually {\it not} a closed immersion. But, by taking the closure of the image of the generic fibre of $\xi^*$, we   obtain a closed subgroup scheme $\mm{Gal}(M)$ of $\bb{GL}(\xi^*M)$. (On the level of Hopf algebras, $\mm{Gal}(M)$ corresponds to the image of the algebra of $\mm{Gal}'(M)$ in that of $\bb{GL}(\xi^*M)$ \cite[4.1]{duong-hai-dos_santos18}.)
This group scheme might be called the {\it restricted differential Galois group of $(M,\na)$}. See \cite{duong-hai-dos_santos18}, specially p.1017. Note that the construction of $\mm{Gal}(M)$ is not Tannakian, but its category of representations on {\it finite and free $R$-modules} might be identified inside $\langle M\rangle_\ot$: In the notation of ($\dagger$), instead of considering all $N$ and all subobjects, we restrict attention to those which are {\it vector bundles} and to certain {\it special} subobjects. 
In summary, and  speaking  loosely, $\mm{Gal}(M)$ depends more on the generic nature of $(M,\na)$ while $\mm{Gal}'(M)$ is more complicated and carries information concerning the reduction of $M$. It is therefore more delicate to pin down ${\rm Gal}'(M)$.  If however $M$ is regular-singular, Corollary \ref{25.11.2019--1} may help us to compute ${\rm Gal}'(M)$.

Let us now assume that $(M,\na)$ is regular-singular.
Proposition \ref{16.01.2019--1} assures that $\langle M\rangle_\ot$  is a {\it full subcategory} of $ \g C^{\rm rs}_Y(X^*/S)$, so that the equivalence in Corollary  \ref{25.11.2019--1} allows us to compute $\mm{Gal}'(M)$ with the help of $\Ga$ in the following way. 

The category $\rep R\Ga$ is Tannakian (in the sense of \cite[Definition 1.2.5]{duong-hai18} and due to \cite[Corollary 4.5]{hai-dos_santos19}) so that \[\rep R\Ga\simeq \rep R\Pi\] for a certain group scheme $\Pi$, called the   {\it Tannakian envelope of $\Ga$} \cite[Definition 4.1]{hai-dos_santos19}. In fact, the abstract group $\Pi(R)$ is the target of an ``universal'' arrow \[u:\Ga\aro\Pi(R)\] having the property that the natural functor \[u^\#:\rep R\Pi\aro\rep R\Ga\] deduced from $u$ is an {\it equivalence} of $R$-linear $\ot$-categories. Moreover, if $G$ is a group scheme and  $\ph:\Ga\to G(R)$ is a morphism of abstract groups, then there exists a unique arrow of group schemes 
\begin{equation}\label{27.01.2022--2}
u_\ph:\Pi\aro G
\end{equation}
such that 
\begin{equation}\label{27.01.2022--1}
\xymatrix{\Ga\ar[dr]_\ph\ar[r]^-{u}&\Pi(R)\ar[d]^{u_\ph (R)}
\\&G(R)}
\end{equation}
commutes.  Furthermore, if  the set  $\{\varphi(\gamma)_{0}\}_{\gamma\in \Gamma}$, respectively $\{\varphi(\gamma)_{\CC(\!(t)\!)}\}_{\gamma\in \Gamma}$, 
 is schematically dense in $G_0$, respectively    in $G_{\CC(\!(t)\!)}$, then $u_\varphi$ is {\it faithfully flat}. 
On the other hand, if   $\{\varphi(\gamma)_0\}_{\gamma\in \Gamma}$  lies in the $\CC$-points of a 
a closed $\CC$-subgroup scheme $H$ of $G_{0}$, then $u$ shall factor through the points of 
the Neron blowup of $G$ with center in $H$. See  \cite[Section 4.2]{hai-dos_santos19}, especially Corollary 4.9 and Proposition 4.13 there.

For any finite and free $R$-module $E$ affording a representation $\ph:\Ga\to\GL(E)$,   let $\langle E\rangle_\ot$ be constructed along the same lines as  ($\dagger$), i.e. $\langle E\rangle_\ot$ is the full subcategory of $\rep R\Ga$ of  subquotients of objects tensor generated by $E$. By Tannakian duality,  the category $\langle E\rangle_\ot$ is equivalent to   $\rep R{\mm{Gal}'(E)}$, where $\mm{Gal}'(E)$ is a group scheme which we call, for the sake of discussion, the {\it full Galois group} of $E$. In addition, the inclusion functor $\langle E\rangle_\ot\to\rep R\Ga$ produces a {\it faithfully flat} arrow \[\Pi\aro\mm{Gal}'(E)\] as assured by \cite[Theorem 4.1.2(i), p.1125]{duong-hai18}. (We note in passing that the same construction applies to any finite and free $R$-module affording a representation of a group scheme.)
Furthermore, if $u_\ph$ is as in \eqref{27.01.2022--2}, then we have a factorization 
\begin{equation}\label{27.01.2022--3}
\xymatrix{\Pi\ar[rrrr]^{u_\ph}\ar[d]_{\text{faithfully flat}} &&&& \bb{GL}(E)
\\
{\rm Gal}'(E)\ar[rrrr]^{\text{isomorphism on}}_{\text{ generic fiber}} &&&& {\rm Gal}(E)\ar[u]_{\text{closed}},
}
\end{equation}
which was called the {\it diptych} in \cite[Definition 4.1]{duong-hai-dos_santos18}. This factorization is, in the correct setting, {\it unique} \cite[Lemma 4.4]{duong-hai-dos_santos18} and serves to understand ${\rm Gal}'(E)$ as ``the image'' of $u_\ph$. It is by means of the diptych that we shall carry out our computations. 

In view of Corollary \ref{25.11.2019--1}, if $M\in\g C_Y^{\rm rs}(X^*/S)$ is as above, then $E=\xi^*M$ is a free $R$-module affording a representation of $\Ga$ and
\begin{equation}\label{31.01.2022--2}
\mm{Gal}'(M)\simeq\mm{Gal}'(E).
\end{equation} 
Using this and the material developed in \cite[\S4.2-3]{hai-dos_santos19}, we shall   exhibit below a simple case where $\mm{Gal}'(M)$ is easily computed and  fails to be of finite type, see Example \ref{31.01.2022--1}.

\subsection{The case of $X=\PP^1_R$ minus two points}\label{26.02.2020--2}
Let $X=\mm {Proj}\,R[x_0,x_1]$ and  write $x=x_1/x_0$, which is  considered as an inhomogeneous coordinate.  
We then fix $Y=0+\infty$ and $\xi=1$. This being so, the $\co_X$-module $\Om^1_{X/R}(\log Y)$ is free on $ x^{-1}dx$ and   $\Ga$ is generated by the loop 
$\ga:s\mapsto e^{\mm is}$. We start with an example given in  \cite[Example~7.9]{duong-hai-dos_santos18}.

\begin{ex}   Let $\cm=\co_X\bb m$ and define on it a logarithmic connection by     
\[
\begin{split}
\na  \bb m & = -\fr{ t}{2\pi\mm i}\bb m\ot\fr{dx}x.
\end{split}
\]
Let $M\in\g C_Y^{\rm rs}(X^*/S)$ be      $\cm|_{X^*}$. We set ourselves to compute $\mm{Gal}'(M)$. From Corollary  \ref{25.11.2019--1}, see eq. \eqref{31.01.2022--2}, our computations can be carried out in  $\rep R\Ga$ and we need first to describe the representation of $\Ga$ associated to $M$. This means that we should work on the complex space $(X_k^*)^\an$ for some unspecified $k$.

The corresponding differential equation is 
\[  y' =\frac t{2\pi\mm ix} y \]
and a  solution is $x^{t/2\pi\mm i}$ (this is a regular function on $(X_k^*)^\an$ near 1). The monodromy, i.e. the representation of $\Ga$ on $R_k=\xi_k^*M_k$  obtained by analytic continuation,   is:
\[ \gamma\longmapsto e^t\in R_k^\ti.\]
Letting $k$ vary, we see that the representation  $\ph:\Ga\to\GL_1(R)$ associated to $M$ is given by $\ga\mapsto e^t$. Let us now adopt the notations of Section \ref{26.02.2020--1}; in particular $\Pi$ is a group scheme   and $u_\ph:\Pi\to \bb{GL}_1$ is a morphism of group schemes. We are required to find the diptych of $u_\ph$, see diagram \eqref{27.01.2022--3}.

We notice that the image of $\Gamma$ is dense    in $\GG_{m}(\CC(\!(t)\!))$ while its image  in $\GG_m(\CC)$ is the unit element.
Thus, $\ph$ factors through the Neron blowup  at the trivial subgroup of the closed fibre, call it $G\to \GG_m$. (The reader will profit to know that $G(R)$ is the ``congruence'' group $\{a\in R^\ti\,:\,a\equiv1\mod t\}$.)  
In a diagram: 
\[
\xymatrix{\Ga\ar[dr]_\ph\ar[r]^-{u}\ar[d]_{\ph'}&\Pi(R)\ar[d]^{u_\ph (R)}\\
 G(R)\ar[r]&\GG_{m }(R).}
\]
Now  $\displaystyle \ph'(\gamma)_0\in G_0(\CC)$ is not trivial and since $G_0=\GG_{a,\CC}$, we conclude that $\ph'(\ga)_0$ generates a dense subgroup.  Hence $u_{\ph'}$ is faithfully flat (as explained in Section \ref{26.02.2020--1}) and $\mm{Gal'}(M)=G$ since 
\[
\xymatrix{
\Pi\ar[r]^-{u_\ph}\ar[d]_{u_{\ph'}}& \GG_{m,R}
\\
G\ar[r]&\GG_{m,R}\ar@{=}[u]
}
\] is the diptych of $u_\ph$. 
\end{ex}

\begin{ex}\label{31.01.2022--1}
 Let $\cm=\co_X\bb m_1\op\co_X\bb m_2$ and define on it a   logarithmic connection by the equations 
\[
\begin{split}
-2\pi\mm i\na  \bb m_1 & =  t\bb m_1\ot\fr{dx}x\\
-2\pi\mm i\na  \bb m_2&= \left( \bb m_1+   t \bb m_2 \right)\ot\fr{dx}x.
\end{split}
\]
Let $M$ denote the object of $\g C_Y^{\rm rs}(X^*/S)$ obtained by restricting $\cm$. We wish to compute $\mm{Gal}'(M)$.

The system of linear differential equations associated to the above connection is
\[ \left\{
\begin{split}
2\pi\mm i y' & = \fr txu+\fr1xv \\
2\pi\mm i z'&= \fr txv.
\end{split}\right.\]
Fixing $k\in\NN$, the columns of  
\[
x^{\fr1{2\pi\mm i}  \begin{pmatrix}t&1\\0&t\end{pmatrix}}:=\exp\left[
\fr{\log x}{2\pi\mm i}\begin{pmatrix}t&1\\0&t\end{pmatrix} \right]
\]
are analytic functions on $(X_k^*)^\an$ near 1 which are solutions of the differential system. 
Therefore, the monodromy matrix, i.e the representation of $\Ga$ on $E_k:=\xi_k^*(M_k)$, corresponds to $\ga\mapsto\begin{pmatrix}e^t&e^t\\0&e^t\end{pmatrix}$. Making $k$ vary, we conclude that the representation $\ph:\Ga\to\GL_2(R)$ associated to $M$ is $\ga\mapsto\begin{pmatrix}e^t&e^t\\0&e^t\end{pmatrix}$ and the computation of $\mm{Gal}'(M)$ can be done in $\rep R\Ga$.  We now bring in the notations of Section \ref{26.02.2020--1}, in particular $\Pi$ a group scheme   and $u_\ph:\Pi\to \bb{GL}_2$ is a morphism of group schemes. 

Let  $\rho:\GG_a\ti\GG_m\to\bb{GL}_2$ be the closed immersion defined by   
\[
(a,\la)\longmapsto\begin{pmatrix}1& a\\0&1\end{pmatrix}\begin{pmatrix}\la&  \\ &\la\end{pmatrix}=\begin{pmatrix}\la&\la a\\0&\la\end{pmatrix}.
\] 
Then, $\ph$ factors through $\ph':\Ga\to\GG_{a}(R)\ti\GG_{m}(R)$, where $\ph'(\ga)=(1,e^t)$. At this point, the reader should observe that the image of $\Ga$ is dense on the generic fiber of $\GG_{a,R}\ti\GG_{m,R}$, while on the special fibre it lies on $\GG_a(\CC)\ti\{1\}$.

In \cite{hai-dos_santos19}, Section 4.3, we constructed a group scheme $\cn$ having the ensuing properties. 

\begin{enumerate}[(i)]\item There exists an arrow $j:\cn\to\GG_a\ti\GG_m$ such that $j\ot_R\CC(\!(t)\!)$ is an isomorphism. 
\item On the level of $R$-points,     $\cn(R)$ is $\{(a,\la)\in R\ti R^\ti\,:\,\la=e^{ta}\}$. In particular, $\ph'$ factors through $j(R)$. 
\item If   $\ph'':\Ga\to\cn(R)$ is defined by $\ga\mapsto(1,e^t)$, so that  $\ph'=j\ph''$,
then the morphism $u_{\ph''}:\Pi\to\cn$ mentioned in \eqref{27.01.2022--1}  is {\it faithfully flat} \cite[Proposition 4.15]{hai-dos_santos19}. 
\end{enumerate}
From (i) and (iii) we conclude that 
\[
\xymatrix{\Pi\ar[r]^{u_\ph}\ar[d]_{u_{\ph''}} & \bb{GL}_2
\\
\cn\ar[r]_-{j} & \GG_a\ti\GG_m\ar[u] 
}
\]
is the diptych of $u_\ph$ and  $\mm{Gal}'(M)\simeq\cn.$ We notice that $\cn$ is {\em not} of finite type (see the discussion in the proof of Corollary 4.16,  top of p. 9399, in \cite{hai-dos_santos19}).

\end{ex}

%Let $E$ be a free $R$-module affording a faithful representation of $\GG_a\ti\GG_m$, say 
%\[
%\rho: \GG_a\ti\GG_m\aro \bb{GL}(E).
%\] 
%From (i) and (iii) we conclude that 
%\[
%\xymatrix{\Pi\ar[r]\ar[d]_{u_\ph} & \bb{GL}(E)
%\\
%\cn\ar[r]_-{j} & \GG_a\ti\GG_m\ar[u]_\rho
%}
%\]
%is the diptych of $\Pi\to\bb{GL}(E)$ \cite[Definition 4.1]{duong-hai-dos_santos18}. Hence, according to \cite[Proposition 4.10]{duong-hai-dos_santos18},   $\rep R\cn=\langle j^\# E\rangle_\ot$. (Here and in what follows, we use the upper-script ``$\#$'' to denote the functor associated to a homomorphism.) 
%Employing the equivalence $u^\#$, we are able to compute the full Galois group of $\ph^\#j^\#(E)\in\rep R\Ga$: 
%\[
%\mm{Gal}'(\ph^\#j^\#E)\simeq\cn.
%\]
%
%Let us explicitly fix $\rho$.  We take $E=R^2$ with an action of $\GG_a\ti\GG_m$ via $(a,\la)\mapsto\begin{pmatrix}1& a\\0&1\end{pmatrix}\begin{pmatrix}\la&  \\ &\la\end{pmatrix}=\begin{pmatrix}\la&\la a\\0&\la\end{pmatrix}$ so that $\ph^\#j^\#E$ is determined by $\ga\mapsto\begin{pmatrix}e^t&e^t\\0&e^t\end{pmatrix}$.
% 
% 

\end{document}